\documentclass[review,hidelinks,onefignum,onetabnum,nohypdvips,nolineno]{siamart220329}

\usepackage{tikz}
\usepackage{booktabs}
\usepackage{siunitx}
\usepackage{subfig}
\usepackage{lipsum}
\usepackage{amsfonts}
\usepackage{graphicx}
\usepackage{epstopdf}
\usepackage{multirow}
\usepackage{diagbox}
\usepackage[margin=1in]{geometry}
\usepackage{algorithmic}
\usepackage{mathtools}
\ifpdf
  \DeclareGraphicsExtensions{.eps,.pdf,.png,.jpg}
\else
  \DeclareGraphicsExtensions{.eps}
\fi

\usepackage[todonotes={textsize=tiny}]{changes}
\definechangesauthor[name={Zhicheng HU}, color=red]{zc}
\definechangesauthor[name={Guanghan LI}, color=orange]{gh}
\setdeletedmarkup{{\color{blue} \sout{#1}}}

\nolinenumbers

\newcommand\dd{\,\mathrm{d}}

\newcommand\bbR{\mathbb{R}}
\newcommand\bbN{\mathbb{N}}

\newcommand\He{\mathit{He}}
\newcommand\Kn{\mathit{Kn}}
\newcommand\Ma{\mathit{Ma}}

\newcommand\mA{\mathcal{A}}

\newcommand\mT{\mathcal{T}}
\newcommand\mH{\mathcal{H}}

\newcommand\mR{\mathcal{R}}

\newcommand\mQ{\mathcal{Q}}

\newcommand\bI{{\bf{I}}}

\newcommand{\bbf}{\boldsymbol{f}}

\newcommand\bxi{\boldsymbol{\xi}}
\newcommand\bx{\boldsymbol{x}}

\newcommand\bg{\boldsymbol{g}}

\newcommand\bj{\boldsymbol{j}}
\newcommand\bq{\boldsymbol{q}}

\newcommand\bu{\boldsymbol{u}}
\newcommand\bU{\boldsymbol{U}}
\newcommand\bv{\boldsymbol{v}}

\newcommand\bi{\boldsymbol{i}}
\newcommand\bn{\boldsymbol{n}}

\newcommand\bF{\boldsymbol{F}}

\newcommand\bvpi{\boldsymbol{\varpi}}
\newcommand\bvrho{\boldsymbol{\varrho}}
\newcommand\bsigma{\boldsymbol{\sigma}}

\newcommand\pd[2]{\dfrac{\partial {#1}}{\partial {#2}}}

\newcommand\opd[2]{\dfrac{\dd {#1}}{\dd {#2}}}

\newcommand\bra[1]{\left( {#1} \right)}
\newcommand\abs[1]{| {#1} |}

\renewcommand{\figurename}{Figure}

\graphicspath{{figs}}

\usetikzlibrary{calc} 
\tikzstyle{start} = [rectangle, draw, fill=pink!20, 
    text width=8em, text centered, rounded corners, minimum height=5em]
\tikzstyle{Euler} = [rectangle, draw, fill=blue!20, 
    text width=8em, text centered, rounded corners, minimum height=5em]
\tikzstyle{Moment} = [rectangle, draw, fill=yellow!20, 
    text width=8em, text centered, rounded corners, minimum height=5em]
\tikzstyle{line} = [draw, -stealth, line width=2pt]

\newsiamremark{remark}{Remark}
\newsiamremark{hypothesis}{Hypothesis}
\crefname{hypothesis}{Hypothesis}{Hypotheses}
\newsiamthm{claim}{Claim}

\headers{A novel fast iterative moment method}{G. Li, C. Wang and Z. Hu}

\title{A novel fast iterative moment method for near-continuum flows
\thanks{Corresponding author: Zhicheng Hu. 
  Submitted to the editors \emph{\today}.
  \funding{This work was partially supported by the National Natural
    Science Foundation of China, No. 12171240, the National Science
    and Technology Major Project, No. J2019-II-0007-0027, the China
    Scholarship Council (CSC) under Grant No. 202406830048 and
    No. 202406830103, the Nanjing University
    of Aeronautics and Astronautics PhD short-term visiting scholar
    project, No. 230901DF08, and Postgraduate Research \& Practice
    Innovation Program of Jiangsu Province, No. KYCX24\_0523.  The
    computational resources were supported by High Performance
    Computing Platform of Nanjing University of Aeronautics and
    Astronautics, China.}}}

\author{Guanghan Li
\and Chunwu Wang
\and Zhicheng Hu \thanks{School of Mathematics, Nanjing University of
Aeronautics and Astronautics, Nanjing 211106, China; Key Laboratory of
Mathematical Modelling and High Performance Computing of Air Vehicles (NUAA),
MIIT, Nanjing 211106, China.\\
\hspace*{2em}\emph{Email addresses}: \email{liguanghan@nuaa.edu.cn} (G. Li);
\email{wangcw@nuaa.edu.cn} (C. Wang); \email{huzhicheng@nuaa.edu.cn} (Z. Hu).}}

\usepackage{amsopn}

\ifpdf
\hypersetup{
  pdftitle={A Novel Fast Iterative Moment Method},
  pdfauthor={G. Li, C. Wang, Z. Hu}
}
\fi

\begin{document}

\maketitle

\begin{abstract}
  We develop a novel fast iterative moment method for the
  steady-state simulation of near-continuum flows, which are modeled by the
  high-order moment system derived from the Boltzmann-BGK equation. The fast
  convergence of the present method is mainly achieved by alternately solving
  the moment system and the hydrodynamic equations with consistent
  constitutive relations and boundary conditions. To be specific, the
  consistent hydrodynamic equations are solved in each alternating iteration to obtain
  improved predictions of macroscopic quantities, which are subsequently 
  utilized to expedite the evolution of the moment system. Additionally, a
  semi-implicit scheme treating the collision term implicitly is introduced
  for the moment system. The resulting alternating iteration can be further
  accelerated by employing the Gauss-Seidel method with a
  cell-by-cell sweeping strategy.
  It is also noteworthy that such an alternating iteration works
  well with the nonlinear multigrid method. Numerical experiments for planar
  Couette flow, shock structure, and lid-driven cavity flow are carried out
  to investigate the performance of the proposed fast iterative moment
  method. All results show impressive efficiency and robustness.  
\end{abstract}



\section{Introduction}
\label{sec:intro}

The rarefaction of gas flow is usually measured by the Knudsen number
$\Kn$, defined as the ratio of the molecular mean free path to the
characteristic flow length. According to its magnitude, gas flows are
classified into four regimes \cite{Wu2022}: continuum for
$\Kn<0.001$, near-continuum or slip for $0.001<\Kn<0.1$, transition
for $0.1<\Kn<10$, and free-molecular for $\Kn>10$. In the continuum
flow regime, classical hydrodynamic equations such as the
Navier-Stokes-Fourier (NSF) equations are sufficient to describe flow
behavior accurately,
whereas in other regimes, they gradually become
inaccurate and may even break down, as rarefaction effects strengthen
with increasing $\Kn$.
In this context, the Boltzmann equation, which is the fundamental
equation of gas kinetic theory and widely considered valid across all
flow regimes, serves as a necessary model for capturing
rarefaction effects. However, solving the Boltzmann equation poses
great computational challenges, due to its intrinsic high dimensional
and nonlinear integro-differential structure. To balance accuracy and
efficiency, a common strategy in the near-continuum flow regime
replaces the complex binary collision integral with a simple
relaxation term, such as the Bhatnagar-Gross-Krook (BGK) model
\cite{Bhatnagar1954}. Nevertheless, even with this simplification,
numerical simulations can still be
computationally expensive, especially for steady-state problems. Over
the past few decades, considerable efforts have been devoted to
developing efficient numerical methods for the Boltzmann-BGK equation,
see, e.g., \cite{Mieussens2000, Xu2010, Cai2010, Yang2022,
  Boscarino2022}.

An attractive approach to reducing computational complexity is to
apply model reduction techniques to yield macroscopic transport
models, often referred to as extended hydrodynamic equations, which
can capture non-negligible rarefaction effects beyond the capability
of classical hydrodynamic equations \cite{Struchtrup2005}. Grad's
moment method \cite{Grad1949} is one of the most prominent techniques
in this direction. It approximates the distribution function in the
Boltzmann equation by a truncated series of Hermite functions,
resulting in a high-order moment system, which, however, may suffer
from ill-posedness and numerical difficulties due to the lack of
hyperbolicity. Building upon this method, a class of globally
hyperbolic regularized moment systems of arbitrary order has been
systematically developed in recent years \cite{Cai2014,
  Cai2015framework}. Accompanying these systems, a unified numerical
framework has been introduced and refined in \cite{Cai2010, Cai2012,
  Cai2013microflow}. This framework enables the tractable and
efficient numerical implementation of high-order moment systems
without the need to explicitly write out each moment equation, thereby
facilitating the practical application of such models to a wide range
of flow problems.

Steady-state solutions are of particular importance in many
applications involving rarefied gas dynamics and related
non-equilibrium phenomena. Conventional numerical methods
based on time integration typically require the solution to evolve
over a prolonged period to reach 
steady state, leading to excessively high computational cost,
especially in the near-continuum flow regime, where their convergence
deteriorates significantly as $\Kn$ decreases.
To address this efficiency bottleneck, attention has
increasingly focused on specific solvers that bypass transient
dynamics and accelerate convergence by directly targeting the
steady-state solution. Among various acceleration strategies,
multigrid methods \cite{Brandt2011, Trottenberg2000} have proven to be
effective by employing coarse level corrections to efficiently
eliminate error components across different scale levels.

Incorporating the nonlinear multigrid (NMG) method into the unified
numerical framework of hyperbolic moment systems, we have developed
an NMG solver based on spatial coarse
grid corrections \cite{Hu2014nmg, Hu2023}, and a nonlinear multi-level
moment (NMLM) solver based on lower-order model corrections
\cite{Hu2016, Hu2019}, which can also be interpreted as coarse level
corrections in velocity space. Both solvers have shown notable
efficiency gains compared to their single level
counterparts. Nevertheless, as with time-integration-based single
level solvers, their convergence for near-continuum flows
still deteriorates markedly as $\Kn$ decreases, thereby weakening the
acceleration effect. Moreover, while each method performs well
individually, integrating both types of coarse level corrections
to develop a more efficient multi-level
solver remains nontrivial. Past attempts in this
direction often encountered issues related to algorithmic stability
and computational overhead. A robust integration strategy has yet to
be established.

Inspired by the effectiveness of the NMLM solver, we propose in this
paper a fast iterative moment (FIM) method that utilizes lower-order
model corrections in a novel and more robust approach. To this end, we
first note that the hydrodynamic equations are included in the moment
system of order $M\geq 2$, and are expected to provide reasonably
reliable predictions of the primary macroscopic quantities, namely
density, mean velocity, and temperature, in the near-continuum
flow regime. 
In contrast, higher-order moments tend to evolve more slowly and often
remain small in this regime. Leveraging this inherent multi-scale
structure, we then introduce the FIM framework, in which the
high-order moment system and the hydrodynamic equations are solved
alternately, with appropriate constitutive relations and boundary
conditions to ensure consistency.
Since the hydrodynamic equations therein effectively serve as the
lower-order model to efficiently predict the primary macroscopic
quantities, the FIM method can be interpreted as a restructured
two-level NMLM method, offering significantly improved robustness and
efficiency, particularly at smaller $\Kn$.
This enables efficient simulation even for moment systems of
moderately large order. Additionally, the present FIM method shares
conceptual similarities with the recently developed GSIS \cite{Su2020,
  Zhu2021}, UGKS \cite{Zhu2016}, and several other micro-macro
acceleration methods \cite{Dong2024, Cai2024SGS, Koellermeier2023,
  Yang2022}.

The FIM method is further improved by incorporating several additional
acceleration strategies. 
Specifically, an implicit treatment of the collision term is
introduced to enhance numerical stability in the small $\Kn$ regime.
Meanwhile, a symmetric Gauss-Seidel method with a cell-by-cell
sweeping strategy is applied to both the moment system and the
hydrodynamic equations to accelerate convergence. In combination with
these favorable strategies, a family of FIM solvers is developed,
exhibiting enhanced performance. Moreover, these solvers can be
readily integrated into the spatial NMG framework to achieve further
acceleration. The resulting FIM-based NMG solver partially addresses
the challenge of developing integrated multi-level solvers that
leverage coarse level corrections in both spatial and velocity spaces.

The remainder of this paper is organized as follows. In
\cref{sec:model}, we provide the necessary preliminaries, including
the fundamental model, i.e., the Boltzmann-BGK equation, and its full
discretization within the unified numerical framework of moment
methods. The core framework of the FIM method is described in detail
in \cref{sec:mac-acc}, while \cref{sec:further-acceleration}
introduces additional targeted strategies, including the implicit
treatment of the collision term, the symmetric Gauss-Seidel method,
and the NMG method. Numerical experiments are presented in
\cref{sec:experiments} to investigate the performance of the proposed
FIM and FIM-based NMG solvers. Finally, a brief conclusion is provided
in \cref{sec:conclusion}.


\section{Preliminaries}
\label{sec:model}

In this section, we provide a brief overview of the governing
Boltzmann-BGK equation and its moment system of arbitrary order,
followed by a unified spatial discretization and a basic time
integration scheme, namely the forward Euler method.

\subsection{Boltzmann-BGK equation}
\label{sec:Boltzmann}

In the near-continuum flow regime, the Boltzmann-BGK equation is
generally sufficient to provide an accurate description of gas
dynamics. It reads
\begin{align}
  \label{eq:boltzmann}
  \pd{f}{t} + \bxi \cdot \nabla_{\bx} f  = \nu (f^{\text{M}} - f),
\end{align}
where $f(t,\bx,\bxi)$ is the molecular distribution function
with $t\in \bbR^{+}$, $\bx \in \Omega \subset \bbR^{D}$ ($D = 1$,
$2$, or $3$), and $\bxi \in \bbR^{3}$ representing time, spatial
position, and molecular velocity, respectively. The left-hand side of
\cref{eq:boltzmann} describes the evolution of $f$ due to the free
streaming of molecules, while the right-hand side accounts for
molecular collisions via a simplified relaxation process. Besides,
$\nabla_{\bx}$ is the gradient operator with respect to $\bx$, and
$\nu$ denotes the average collision frequency, assumed to be
independent of $\bxi$ and inversely proportional to the Knudsen number,
i.e., $\nu \propto 1/\Kn$. For the BGK collision term,
$f^{\text{M}}$ is the local Maxwellian defined as
\begin{align}
  \label{eq:maxwellian}
  f^{\text{M}}(t,\bx,\bxi) = \frac{\rho(t, \bx)}{m_* \left(2 \pi
  \theta(t,\bx)\right)^{3/2}} \exp \left( -
  \frac{\abs{\bxi - \bu(t, \bx)}^2}{2 \theta(t, \bx)} \right),
\end{align}
in which $m_{*}$ is the mass of a single molecule, and $\rho$,
$\bu$, and $\theta$ are density, mean
velocity, and temperature, respectively. These primary macroscopic
quantities are determined by $f$ through the following relations
\begin{align}
  \label{eq:rho-u-theta}
  \begin{aligned}
    & \rho(t,\bx) = m_* \int_{\bbR^3} f(t,\bx,\bxi) \dd \bxi, \\ 
    & \rho(t,\bx) \bu(t,\bx) = m_* \int_{\bbR^3} \bxi f(t,\bx,\bxi) \dd \bxi, \\
    & \rho(t,\bx) \vert \bu(t,\bx) \vert^2 + 3 \rho(t,\bx) \theta(t,\bx) = m_*
      \int_{\bbR^3} \vert \bxi \vert^2 f(t,\bx, \bxi) \dd \bxi.
  \end{aligned}
\end{align}
Two additional macroscopic quantities of interest, namely
the stress tensor $\bsigma$ and heat flux $\bq$, are given by
\begin{align}
  \label{eq:sigma-q}
  \begin{aligned}
  & \sigma_{ij}(t,\bx) = m_* \int_{\bbR^3} (\xi_i - u_i(t,\bx))(\xi_j -
  u_j(t,\bx)) f(t,\bx,\bxi) \dd \bxi - \rho(t,\bx) \theta(t,\bx) \delta_{ij},
  \\ & \bq(t,\bx) = \frac{m_*}{2} \int_{\bbR^3} \vert
  \bxi-\bu(t,\bx) \vert^2 (\bxi-\bu(t,\bx)) f(t,\bx,\bxi) \dd \bxi,
  \end{aligned}
\end{align}
where $\delta_{ij}$ is the Kronecker delta symbol, and $i, j = 1, 2, 3$
represent the Cartesian coordinate directions.

\subsection{Grad's expansion and hyperbolic moment system}
\label{sec:model-vdis}

In Grad's moment method, the distribution function $f$ is expanded
into a series in terms of Hermite functions as
\begin{align}
  \label{eq:dis-expansion}
  f(t, \bx, \bxi) = \sum_{\alpha \in \bbN^{3}} f_\alpha(t, \bx)
  \mH_{\alpha}^{[\bu, \theta]} (\bxi), 
\end{align}
where $f_{\alpha}(t,\bx)$ are the expansion coefficients representing
$\alpha$th-order moments of $f$, 
and $\mH_{\alpha}^{[\cdot, \cdot]}(\cdot)$ are the corresponding
basis functions defined as
\begin{align*}
  \mH_{\alpha}^{[\bu, \theta]}(\bxi) =
  \frac{1}{m_*(2\pi\theta)^{3/2}\theta^{|\alpha|/2}} \prod\limits_{d=1}^3
  \He_{\alpha_d}({v}_d)\exp \left(-\frac{{v}_d^2}{2} \right),
  \quad {\bv} = \frac{\bxi-\bu}{
    \sqrt{\theta}},~~\forall \bxi \in \bbR^3,
\end{align*}
in which $|\alpha| = \alpha_1 + \alpha_2 + \alpha_3$ is the sum of all
components of the three-dimensional multi-index
$\alpha = (\alpha_1, \alpha_2, \alpha_3) \in \bbN^3$, and
$\He_n(\cdot)$ is the Hermite polynomial of degree $n$, namely,
\begin{align*}
  \He_n(x) = (-1)^n\exp \left(\frac{x^{2}}{2} \right) \frac{\dd^n}{\dd x^n} 
  \exp \left(-\frac{x^{2}}{2}\right).
\end{align*}
Using the orthogonality of Hermite polynomials, we can easily deduce
that all basis functions $\mH_{\alpha}^{[\bu, \theta]}(\bxi)$ are
orthogonal to each other over $\bbR^3$ with respect to the weight
function $\exp(|\bv|^{2}/2)$, and the expansion coefficients are given by
\begin{align*}
  f_{\alpha}(t,\bx) = \frac{m_{*}^{2}(2\pi\theta)^{3/2}\theta^{|\alpha|}}{\alpha_{1}!\alpha_{2}!\alpha_{3}!} \int f(t,\bx,\bxi) \mH_{\alpha}^{[\bu,\theta]}(\bxi) \exp\left(\frac{\vert \bv \vert^{2}}{2}\right) \dd \bxi.
\end{align*}
Then, from relations \cref{eq:rho-u-theta,eq:sigma-q} between
macroscopic quantities and the distribution function, we have
\begin{align}
  \label{eq:moments-relation}
  \begin{aligned}
    &f_0 = \rho, \qquad f_{e_1} = f_{e_2} = f_{e_3} = 0, \qquad
    \sum_{d=1}^3 f_{2e_d} = 0,\\
    &\sigma_{ij} = (1+\delta_{ij}) f_{e_i+e_j}, \quad q_i = 2
    f_{3e_i} + \sum_{d=1}^3 f_{2e_d+e_i}, \qquad i,j=1,2,3,
  \end{aligned}
\end{align}
where $e_1$, $e_2$, and $e_3$ correspond to the three-dimensional
multi-indices $(1,0,0)$, $(0,1,0)$, and $(0,0,1)$, respectively. It follows
that the leading term in the expansion \cref{eq:dis-expansion} exactly
recovers the local Maxwellian, that is,
$f^{\text{M}} = f_{0} \mH_{0}^{[\bu,\theta]}(\bxi)$, which implies
that Grad's expansion provides an efficient representation of the
distribution function
for flows close to thermodynamic equilibrium. However,
in practice, it is often convenient to expand $f$
in terms of Hermite functions with
generalized parameters $\bvpi$ and $\vartheta$, respectively, in place
of the local mean velocity $\bu$ and temperature $\theta$ used in
Grad's expansion. In this case, relations similar to
\cref{eq:moments-relation} can be derived accordingly, indicating that
all the aforementioned macroscopic quantities of interest can be
determined by $\bvpi$, $\vartheta$, and the low-order coefficients
$f_{\alpha}$ with $|\alpha|\leq 3$. We refer to \cite{Hu2019} for the
detailed expressions and omit them here for brevity.

Applying the systematic approach of the hyperbolic moment model reduction
proposed in \cite{Cai2014,Cai2015framework} to the Boltzmann-BGK
equation \cref{eq:boltzmann}, together with Grad's expansion
\cref{eq:dis-expansion}
and a selected truncation order $M$, we obtain a system of
equations for $\bu$, $\theta$, and $f_{\alpha}$ with
$|\alpha| \leq M$, which is referred to as the moment system of order $M$, and
can be rewritten into a quasi-linear form as
\begin{align}
  \label{eq:moment-system}
	 \pd{\bvrho}{t} + \sum_{j=1}^{D} {\bf{A}}_{j}(\bvrho) \pd{\bvrho}{x_{j}}
	 = \mQ(\bvrho),
\end{align}
by introducing $\bvrho$ as a vector that collects all unknown
variables. For conciseness and without loss of clarity, its
component-wise expression is not shown explicitly in this work.
Nevertheless, several important properties of this moment system
are worth mentioning. First, it is globally hyperbolic with
all eigenvalues given by $\bu\cdot \bn + c_{ij} \sqrt{\theta}$ for any
matrix of the form
$n_{1} {\bf{A}}_{1}(\bvrho)+\cdots + n_{D} {\bf{A}}_{D}(\bvrho)$,
where $\bn=(n_{1},\ldots,n_{D})^{T}$ is a unit vector, and $c_{ij}$
denotes the $i$th root of the Hermite polynomial $\He_{j}(x)$, with
$j=1,2,\ldots,M+1$, and $i=1,2,\ldots,j$. Second, its solution
shall provide an $M$th-order approximation of the distribution function,
i.e.,
\begin{align}
  \label{eq:truncated-dis}
  f(t, \bx, \bxi) \approx \sum_{|\alpha| \leq M} f_\alpha(t, \bx)
  \mH_{\alpha}^{[\bu, \theta]} (\bxi),
\end{align}
which indicates that the moment system can be interpreted as a
semi-discretization of the Boltzmann-BGK equation in velocity
space. This allows us to construct unified numerical solvers for the
moment system of arbitrary order directly based on the underlying
Boltzmann-BGK equation.

Moreover, the moment system is often regarded as an extended
hydrodynamic model when $M\geq 2$, 
since its lower-order components include the hydrodynamic
equations. In particular, the first component, corresponding to the
index $\alpha=0$, turns out to be exactly the continuity equation
\begin{align}
  \label{eq:mac-continuity}
  \pd{\rho}{t} + \nabla_{\bx} \cdot \bra{\rho \bu} = 0.
\end{align}
The next three equations, with the indices $\alpha=e_{i}$, $i=1,2,3$,
correspond to the momentum equations
\begin{align}
  \label{eq:mac-momentum}
  \pd{\bra{\rho \bu}}{t} + \nabla_{\bx} \cdot \bra{\rho \bu \otimes \bu
    + p \bI + \bsigma} = 0,
\end{align}
where $p=\rho\theta$ is the pressure, and $\bI$ is the identity
matrix. Finally, by summing the equations for $\alpha=2e_{i}$,
$i=1,2,3$, we can recover the energy equation
\begin{align}
  \label{eq:mac-energy}
  \pd{\bra{\rho E}}{t} + \nabla_{\bx} \cdot \bra{\rho E \bu + p \bu + \bu
  \cdot \bsigma + \bq} = 0,
\end{align}
in which $E = \frac{1}{2} \left(3 \theta + |\bu|^{2}\right)$
represents the total energy.

\subsection{Spatial discretization and time integration scheme}
\label{sec:model-xdis}

Let $\Omega$ be a rectangular domain in $\bbR^{D}$, partitioned by a
uniform Cartesian grid with $N_{1}\times \cdots \times N_{D}$ cells.
Suppose the grid cell and its center are denoted by
$\mT_{\bi} = [x_{i_1 - 1/2}, x_{i_1 + 1/2}] \times \cdots \times
[x_{i_D - 1/2}, x_{i_D + 1/2}]$ and
$\bx_{\bi}=(x_{i_1}, \ldots, x_{i_D})$, respectively, where
$\bi=(i_{1},i_{2},\ldots,i_{D})$ is the multi-index with
$i_{d}=1,2,\ldots,N_{d}$ and $d=1,\ldots,D$. To derive a unified
spatial discretization for the moment system \cref{eq:moment-system}
of arbitrary order, we return to the Boltzmann-BGK equation
\cref{eq:boltzmann}, and integrate it over the $\bi$th grid
cell $\mT_{\bi}$. Then, introducing $f_{\bi}(t,\bxi)$ and
$\bar{F}(f_{\bi},f_{\bi+e_{d}})$, respectively, to represent the
cell-averaged distribution function over $\mT_{\bi}$ and the numerical
flux across the interface between $\mT_{\bi}$ and its neighboring cell
$\mT_{\bi+e_{d}}$, we obtain a finite volume discretization of the
Boltzmann-BGK equation, formally written as
\begin{align}
  \label{eq:ode-system}
  \opd{f_{\bi}}{t} = - \sum \limits_{d=1}^{D} \frac{1}{\Delta x_{d}}
  \left[ \bar{F} \bra{f_{\bi}, f_{\bi + e_d}} - \bar{F} \bra{f_{\bi-e_d},
  f_{\bi}} \right] + Q(f_{\bi}) \eqqcolon \mR_{\bi}(\bbf),
\end{align}
where $\Delta x_d = x_{i_d + 1/2} - x_{i_d - 1/2}$ is the grid size in
the $d$th direction, $Q(f_{\bi})$ describes the collision term, and
$\mR_{\bi}(\bbf)$ denotes the local residual associated with
$\mT_{\bi}$, in which $\bbf$ is the collection of all $f_{\bi}$.

To proceed, let Grad's expansion of order $M$ be employed for all 
cell-averaged
distribution function, i.e.,
\begin{align}
  \label{eq:dis-i-expansion}
  f_{\bi} (t, \bxi) \approx \sum_{|\alpha| \leq M} f_{\bi,\alpha} (t)
    \mH_{\alpha}^{[\bu_{\bi}, \theta_{\bi}]} (\bxi).
\end{align}
With such an expansion, the BGK collision term directly yields
\begin{align}
  \label{eq:bgk-i-expansion}
  Q(f_{\bi}) \approx - \sum_{2 \leq |\alpha| \leq M} \nu_{\bi} f_{\bi,\alpha} (t)
    \mH_{\alpha}^{[\bu_{\bi}, \theta_{\bi}]} (\bxi).
\end{align}
It is also straightforward to show that numerical fluxes, through a
suitable projection procedure, can be similarly expressed in terms of
the same basis functions, e.g.,
\begin{align}
  \label{eq:flux-expansion}
  \bar{F} \bra{f_{\bi}, f_{\bi + e_d}} \approx \sum_{|\alpha| \leq M}
  \bar{F}_{\alpha} \bra{f_{\bi}, f_{\bi + e_d}} \mH_{\alpha}
  ^{[\bu_{\bi}, \theta_{\bi}]} (\bxi).
\end{align}
Consequently, substituting these expansions into \cref{eq:ode-system},
and matching the coefficients for each basis function, 
we get a nonlinear system with respect to $\bu_{\bi}$, $\theta_{\bi}$,
and $f_{\bi,\alpha}$ for $|\alpha|\leq M$ and $\bi$ ranging over all
grid cells in $\Omega$. Provided that numerical fluxes are constructed
consistently with the hyperbolic moment model reduction, this
nonlinear system shall constitute a unified spatial discretization for
the moment system \cref{eq:moment-system}. 

In the current work, the numerical flux introduced in
\cite{Cai2013microflow} is adopted. It is essentially a modified
HLL flux, with a regularization term
accounting for the hyperbolicity of the system. More precisely,
we have
\begin{align}
  \label{eq:flux-split}
  \bar{F} \bra{f_{\bi}, f_{\bi+e_d}} = \hat{F}
_{\bi+\frac{1}{2}e_d} + \tilde{F}_{\bi+\frac{1}{2}e_d}^{-}, \qquad \bar{F}
\bra{f_{\bi-e_d}, f_{\bi}} = \hat{F}_{\bi-\frac{1}{2}e_d} + \tilde{F}
_{\bi-\frac{1}{2}e_d}^{+},
\end{align}
where $\hat{F}_{\bi+\frac{1}{2}e_d}$ is the HLL flux defined as
\begin{align}
  \label{eq:hll-flux}
  \hat{F}_{\bi+\frac{1}{2}e_d} = 
  \left\{
  \begin{aligned}
    & \xi_d f_{\bi}(t,\bxi), && \lambda_{d}^{L} \geq 0, \\
    & \frac{\lambda_d^R \xi_d f_{\bi}(t,\bxi) - \lambda_d^L \xi_d
      f_{\bi + e_d}(t,\bxi) + \lambda_d^R \lambda_d^L \bra{f_{\bi + e_d}(t,\bxi) - f_{\bi}(t,\bxi)}}
      {\lambda_d^R - \lambda_d^L}, && \lambda_{d}^{L} \textless 0 \textless
                                     \lambda_d^R, \\
    & \xi_d f_{\bi + e_d}(t,\bxi), && \lambda_d^R \leq 0, \\
  \end{aligned}\right.
\end{align}
and $\tilde{F}_{\bi\pm\frac{1}{2}e_d}^{\mp}$ denote the regularization
terms whose expressions are omitted for simplicity. Here,
$\lambda_{d}^{L}$ and $\lambda_{d}^{R}$ are, respectively, the minimal
and maximal eigenvalues of ${\bf{A}}_{d}(\bvrho)$, estimated from the
states of both the $\bi$th and $(\bi+e_{d})$th cells as
\begin{align}
  \label{eq:hll-flux-lambda}
  \lambda_{d}^{L} = \min_{\bj\in\{\bi,\bi+e_{d}\}}\left\{u_{\bj,d} - c_{M}\sqrt{\theta_{\bj}}\right\}, \quad
  \lambda_{d}^{R} = \max_{\bj\in\{\bi,\bi+e_{d}\}}\left\{u_{\bj,d} + c_{M}\sqrt{\theta_{\bj}}\right\},
\end{align}
in which $c_{M}$ is the largest root of the Hermite polynomial
$\He_{M+1}(x)$.

Obviously, appropriate boundary conditions on the domain boundary
$\partial\Omega$ are required to complete the spatial discretization.
In most of the experiments presented in \cref{sec:experiments}, we
consider the Maxwell boundary condition, whose detailed treatment
within the framework of Grad's expansion can be found in
\cite{Cai2012,Cai2013microflow}. An exception is the shock structure
problem, where a Maxwellian with prescribed boundary density, mean
velocity, and temperature would be directly imposed.

Now we turn to the time integration for the system of
semi-discretization \cref{eq:ode-system}. We shall first consider the
forward Euler scheme, one of the simplest time-integration schemes,
taking the form
\begin{align}
  \label{eq:euler-scheme}
  f_{\bi}^{n+1} = f_{\bi}^{n} + \Delta t \mR_{\bi}(\bbf^{n}),
\end{align}
where the superscript $n$ is used to label the approximation of
variables at time $t^{n}$, with the time step size given by
$\Delta t = t^{n+1} - t^{n}$. Due to stability constraints, the
time step size is selected according to the CFL condition, that is, 
\begin{align}
  \label{eq:dt}
  \Delta t = \min_{\bi} \Delta t_{\bi},
\end{align}
where the local time step size $\Delta t_{\bi}$ should satisfy
\begin{align}
  \label{eq:dt-local}
  \Delta t_{\bi} \sum_{d=1}^{D} \frac{\left| u_{\bi,d}^{n}\right| + c_{M}
  \sqrt{\theta_{\bi}^{n}}}{\Delta x_{d}} = \text{CFL} < 1,
\end{align}
with $\text{CFL}$ being a user-specified constant.

We conclude this section by noting that, using Grad's expansion for
each distribution function, the left- and right-hand sides of
\cref{eq:euler-scheme} are expressed in terms of basis functions with
possibly different local parameters. As a result, the scheme
\cref{eq:euler-scheme} numerically involves two steps, as explained in
\cite{Hu2016}.


\section{Fast iterative moment method}
\label{sec:mac-acc}

In this paper, we are primarily concerned with the steady-state
solution to the high-order moment system, or equivalently, the
Boltzmann-BGK equation, in the regime with small $\Kn$,
as $t\to\infty$. Dropping the time derivatives in
\cref{eq:ode-system}, we arrive at the discrete steady-state problem
\begin{align}
  \label{eq:discrete-ssp}
  \mR_{\bi}(\bbf) = 0, 
\end{align}
for all $\bi$. At this stage, the forward Euler scheme
\cref{eq:euler-scheme} serves as a basic iteration, also known as the
Richardson iteration. 
However, once the moment system is of order $M\geq 2$, such that $\Kn$
appears in the higher-order equations via the collision frequency
$\nu$, the convergence of time-integration-based methods, including
the forward Euler scheme, becomes extremely slow as $\Kn$ decreases
toward the continuum limit.
To address such poor convergence, we
shall introduce a novel alternating iteration, referred to as the fast
iterative moment (FIM) method, in this section.

\subsection{Alternating iteration framework}
\label{sec:FIM-framework}

It is well known that the fluid dynamic limit of the Boltzmann
equation leads to the hydrodynamic Euler or NSF
equations as $\Kn$ tends to zero. When $\Kn<0.001$, gas flows fall
into the continuum regime, and the classical hydrodynamic equations
are typically accurate in describing gas dynamics. It follows that the
higher-order moments $f_{\alpha}$ with $|\alpha|\geq 2$ would
be negligible, or otherwise well captured by the primary macroscopic
quantities $\rho$, $\bu$, and $\theta$. The same holds for the stress
tensor $\bsigma$ and heat flux $\bq$. In contrast, when $\Kn>0.001$
but remains relatively small, that is, in the near-continuum flow regime,
the classical hydrodynamic equations may no longer be adequate, and a
high-order moment system involving more equations
becomes necessary. Nevertheless, the classical model can still be
expected to provide fairly reliable approximations of the primary
macroscopic quantities.
Moreover, the order $M$ of the moment system does not need to be very
large, and the rest of the higher-order moments $f_{\alpha}$, as well
as $\bsigma$ and $\bq$, are anticipated either to remain small or vary
slowly during time integration.

Motivated by the above observations, and noting that the hydrodynamic
equations \cref{eq:mac-continuity,eq:mac-momentum,eq:mac-energy} can,
in some sense, be regarded as a subsystem extracted from the full
moment system,
it is natural to expect that this subsystem could be
solved more efficiently
and thereby utilized to accelerate convergence toward the steady-state
solution of the full moment system. Building upon this idea, we
consequently propose a framework of alternating iteration between the
high-order moment system (HMS) and the hydrodynamic equations (HEs),
which forms the core of the resulting FIM method. A schematic
flowchart of this FIM alternating iteration is illustrated in
\cref{fig:flowchart}.

Specifically, starting from an initial guess for the solution $\bbf$ of the HMS,
we first perform a few basic iterations, for example, one or two steps
of the forward Euler scheme \cref{eq:euler-scheme}, on the HMS
to update $\bbf$. The primary macroscopic quantities $\rho$, $\bu$,
and $\theta$ are then extracted from $\bbf$ and updated by solving the
HEs, subject to consistent $\bsigma$, $\bq$, and hydrodynamic boundary
conditions, all evaluated also from the current
$\bbf$. The updated $\rho$, $\bu$, and $\theta$ are subsequently fed back into
$\bbf$, which continues to be updated by repeatedly solving the HMS
and HEs, in turn as described above, until convergence is achieved.
In the next subsections, the main ingredients of the hydrodynamic part in the alternating
iteration, including a numerical scheme for the HEs and a
formulation ensuring discretization consistency
with the HMS, will be described in detail, followed by
a complete FIM algorithm.

\begin{figure}[!htbp]
  \centering
  \begin{tikzpicture}[node distance = 2cm, auto, scale=0.7, transform shape]
    \node [start] (init) {Initial guess for $\bbf$};
    \node [Moment, right of=init, node distance=4.5cm] (moment-1) {Solve HMS to update $\bbf$, step $n=1$};
    \node [Euler, below of=moment-1, node distance=3cm] (euler-1)
    {Solve HEs to update $\rho$, $\bu$, and $\theta$};
    \node [Moment, right of=moment-1, node distance=4.75cm] (moment-2) {Solve HMS to update $\bbf$, step $n=2$};
    \node [Euler, below of=moment-2, node distance=3cm] (euler-2)
    {Solve HEs to update $\rho$, $\bu$, and $\theta$};
    \node [right of = moment-2, node distance=3.5cm] (ellipsis) {$\cdots$};
    \path [line] (init) -- (moment-1);
    \path [line, blue] (moment-1) -- (moment-2);
    \path [line, blue] (moment-2) -- (ellipsis);
    \path [line, red, line width=1.4pt, align=center] (moment-1) -- node [right=0pt] {\small $\rho$, $\bu$, $\theta$, \\ \small $\bsigma$, $\bq$, BC} (euler-1);
    \path [draw, gray, line width=1.4pt, align=center] (euler-1.east) to [out=0,in=180]
    node[right] {\small $\rho$, $\bu$, $\theta$} ($(moment-2.west)-(0.25cm,0)$);
    \path [line, red, line width=1.4pt, align=center] (moment-2) -- node [right=0pt] {\small $\rho$, $\bu$, $\theta$, \\ \small $\bsigma$, $\bq$, BC} (euler-2);
    \path [draw, gray, line width=1.4pt, align=center] (euler-2.east) to [out=0,in=180]
    node[right] {\small $\rho$, $\bu$, $\theta$} ($(ellipsis.west)-(0.25cm,0)$);
  \end{tikzpicture}
  \caption{Flowchart of the alternating iteration in the FIM
    method. HMS: high-order moment system; HEs: hydrodynamic
    equations; BC: consistent hydrodynamic boundary conditions.}
  \label{fig:flowchart}
\end{figure}

\subsection{Numerical scheme for hydrodynamic equations}
\label{sec:FIM-hydro}

The hydrodynamic equations
\cref{eq:mac-continuity,eq:mac-momentum,eq:mac-energy} can be
rewritten in vector form as
\begin{align}
  \label{eq:hydro-eqs-vec}
  \pd{U}{t} + \nabla_{\bx} \cdot \bF \bra{U} = 0,
\end{align}
where
\begin{align*}
  U =
  \begin{bmatrix}
    \rho \\ \rho \bu \\ \rho E
  \end{bmatrix}, \qquad
  \bF(U) = 
  \begin{bmatrix}
    \rho \bu \\ \rho \bu \otimes \bu + p \bI + \bsigma \\
    \rho E \bu + p \bu + \bu \cdot \bsigma + \bq 
  \end{bmatrix}.
\end{align*}
Applying the finite volume method to \cref{eq:hydro-eqs-vec} over the
$\bi$th grid cell $\mT_{\bi}$, we obtain
\begin{align}
  \label{eq:mac-dis}
  \opd{U_{\bi}}{t} = - \sum_{d=1}^{D} \frac{1}{\Delta x_{d}} \left[
  \boldsymbol{\bar{F}} \bra{U_{\bi}, U_{\bi + e_d}} - \boldsymbol{\bar{F}}
  \bra{U_{\bi-e_d}, U_{\bi}} \right] \eqqcolon \mR^{\text{hydro}}_{\bi}(\bU),
\end{align}
where $U_{\bi}$ is the cell average of $U$ associated with
$\mT_{\bi}$, and $\mR^{\text{hydro}}_{\bi}(\bU)$ denotes the
corresponding local residual of the hydrodynamic equations, with $\bU$
representing the collection of all $U_{\bi}$. Following the treatment
for the moment system, the HLL flux is also employed for the numerical
flux $\boldsymbol{\bar{F}}\bra{\cdot,\cdot}$. Thus, we have
\begin{equation}
 \label{eq:hydro-hll-flux}
 \boldsymbol{\bar{F}} \bra{U_{\bi}, U_{\bi + e_d}} = 
 \left\{
  \begin{aligned}
    & \bF\bra{U_{\bi}} , && \lambda_{d}^{L} \geq 0,\\
    & \frac{\lambda_d^R \bF\bra{U_{\bi}} - \lambda_d^L \bF\bra{U_{\bi+e_d}} + \lambda_d^R
      \lambda_d^L \bra{U_{\bi + e_d} - U_{\bi}}}{\lambda_d^R - \lambda_d^L}, &&
   \lambda_{d}^{L} < 0 < \lambda_d^R, \\
    & \bF\bra{U_{\bi + e_d}} , && \lambda_d^R \leq 0. \\
  \end{aligned}\right.
\end{equation}
Here it is important to emphasize that $\lambda_{d}^{L}$ and
$\lambda_{d}^{R}$ must be evaluated using the same expressions as in
\cref{eq:hll-flux-lambda}, which is essential for ensuring consistency
with the spatial discretization \cref{eq:discrete-ssp} of the moment
system.

The semi-discretization \cref{eq:mac-dis} can be solved
similarly using the forward Euler scheme, which yields
\begin{equation}
  \label{eq:hydro-richardson}
  U_{\bi}^{m+1} = U_{\bi}^m + \Delta t \mR_{\bi}^{\text{hydro}}(\bU^m),
\end{equation}
where the time step size $\Delta t$ is determined again by the CFL
condition given in \cref{eq:dt,eq:dt-local}. In each alternating
iteration, this scheme might be executed over multiple time steps,
beginning with an initial guess $\bU^{0}$ obtained from the current
$\bbf$, until the numerical solution $\bU^{m}$ has sufficiently
evolved or converged.

\subsection{Consistent constitutive and boundary conditions}
\label{sec:FIM-cr-bc}

In order to guarantee that the hydrodynamic equations
\cref{eq:hydro-eqs-vec} give rise to steady-state macroscopic
quantities $\rho$, $\bu$, and $\theta$ exactly matching those
obtained from the high-order moment system, it remains necessary to
impose consistent constitutive relations and boundary conditions for
the hydrodynamic equations. To this end, we adopt the following
simplified approaches in the present FIM method.

As in the NSF equations, the stress tensor $\bsigma$
and heat flux $\bq$ that appear in \cref{eq:hydro-eqs-vec} can be
determined from the primary macroscopic quantities $\rho$, $\bu$, and
$\theta$. To avoid additional modeling assumptions, however, they are
instead evaluated from the solution $\bbf$ of the moment system and
kept unchanged during the integration \cref{eq:hydro-richardson} in
each alternating iteration.

Similarly, Dirichlet-type hydrodynamic boundary conditions are
prescribed, rather than being directly derived from the primary
macroscopic quantities. Precisely speaking, the ghost-cell
distribution function $f^{g}$ outside the domain $\Omega$ is first
constructed from the solution $\bbf$ to satisfy the boundary
conditions of the moment system. The macroscopic quantities
$\rho^{g}$, $\bu^{g}$, and $\theta^{g}$, as well as $\bsigma^{g}$ and
$\bq^{g}$, are then evaluated from $f^{g}$ and held fixed during the
integration \cref{eq:hydro-richardson} until the next alternating
iteration.  With these ghost-cell macroscopic
quantities available, all numerical fluxes \cref{eq:hydro-hll-flux}
across the domain boundary can be fully determined, thereby completing
the specification of the hydrodynamic boundary conditions.

\subsection{Complete FIM algorithm}
\label{sec:FIM-algorithm}

Incorporating all the ingredients described above, the framework for
one step of the FIM alternating iteration, which advances the solution
from a given approximation $\bbf^n$ to a new one $\bbf^{n+1}$, is
summarized in \cref{alg:FIM}. Therein, $\mathit{Tol}$ is a prescribed
tolerance indicating the convergence criterion of the hydrodynamic
solver, while $\gamma_{1}$ and $\gamma_{2}$ are user-defined positive
integers controlling the number of iterations for the moment and
hydrodynamic solvers, respectively. All of them are chosen
to balance the computational cost per
alternating iteration with the convergence rate.

Performing the algorithm until convergence immediately yields a
steady-state solver for the Boltzmann-BGK equation. Throughout the
paper, both this solver and the algorithm itself are conveniently
referred to as the FIM-1 solver.
Notably, the moment and hydrodynamic solvers in the algorithm can, in
principle, be replaced by alternatives to obtain further improved FIM solvers.
Several such strategies will be discussed in the
next section, while this section closes with two remarks on the
proposed FIM method.

\begin{remark}
  In comparison to the NMLM solver proposed in
  \cite{Hu2016,Hu2019}, which accelerates convergence by using
  lower-order model corrections, the FIM solver can be
  regarded as a two-level NMLM solver, where the hydrodynamic
  equations \cref{eq:hydro-eqs-vec} serve as the lower-order
  model. However, instead of the restriction and prolongation
  operators designed under the original NMLM framework, a new
  alternative strategy for transferring information between the high-
  and lower-order models is introduced. As demonstrated in
  \cref{sec:experiments}, this redesigned two-level NMLM solver,
  namely, the FIM solver, improves both robustness and efficiency,
  especially in the near-continuum flow regime.
\end{remark}

\begin{remark}
  The FIM method also shares, to some extent, a similar philosophy and
  convergence behavior with the GSIS
  introduced in \cite{Su2020,Zhu2021}, as both
  methods find the steady-state solution through alternating
  iterations between the Boltzmann and hydrodynamic equations.
  However, unlike the GSIS which is typically developed within the
  discrete velocity framework, the present FIM method is formulated in
  the framework of the moment method. Moreover, since the hydrodynamic
  equations arise as a subsystem of the moment system, the FIM method
  requires consistent spatial discretizations for both systems to
  ensure that they yield identical numerical values of the primary
  macroscopic quantities. This consistency requirement is crucial for
  guaranteeing the convergence of the overall alternating iteration.
\end{remark}

\begin{algorithm}[!htbp]
  \renewcommand{\algorithmicrequire}{\textbf{Input:}}
  \renewcommand{\algorithmicensure}{\textbf{Output:}}
  \renewcommand{\algorithmicloop}{\textbf{begin}}
  \renewcommand{\algorithmicendloop}{\algorithmicend}
  \caption{One step of the FIM alternating iteration $\bbf^{n+1} = \text{FIM}\bra{\bbf^{n}}$} 
  \label{alg:FIM}\footnotesize
  \begin{algorithmic}[1]
    \REQUIRE Grid, $\bbf^{n}$ 
    \ENSURE The new approximation $\bbf^{n+1}$ 
    \vspace*{0.7em}
    \STATE Update $\bbf$ from $\bbf^{n}$ by performing $\gamma_1$ steps of a basic moment solver, e.g., the forward Euler scheme \cref{eq:euler-scheme}, denoted as $\boldsymbol{\tilde{f}} = \text{MSolver}^{\gamma_{1}}\left(\bbf^{n}\right)$;
    \STATE Compute $\tilde{\rho}$, $\boldsymbol{\tilde{u}}$, $\tilde{\theta}$, $\boldsymbol{\tilde{\sigma}}$, and $\boldsymbol{\tilde{q}}$ on all grid cells, and construct $f^{g}$ on all ghost cells from $\boldsymbol{\tilde{f}}$;
    \STATE Construct the initial guess $\bU^{0}$ with $U_{\bi}^{0} = \left[\tilde{\rho}_{\bi}, \tilde{\rho}_{\bi}\boldsymbol{\tilde{u}}_{\bi}, \frac{1}{2}\tilde{\rho}_{\bi}\left(3\tilde{\theta}_{\bi}+\vert \boldsymbol{\tilde{u}}_{\bi}\vert^{2}\right)\right]^{T}$ for the hydrodynamic equations;
    \STATE Fix the stress tensor and heat flux in the hydrodynamic equations by $\boldsymbol{\tilde{\sigma}}$ and $\boldsymbol{\tilde{q}}$, respectively;
    \STATE Compute ${\rho}^{g}$, $\bu^{g}$, $\theta^{g}$, ${\bsigma}^{g}$, and ${\bq}^{g}$ on all ghost cells from the corresponding $f^{g}$;
    \STATE Set $m = 0$, and ${\it diff} = 1$;
    \WHILE{$m < \gamma_{2}$ and ${\it diff} > \mathit{Tol}$}
    \STATE Calculate $\bU^{m+1}$ from $\bU^{m}$ by performing one step of a hydrodynamic solver, e.g., the forward Euler scheme \cref{eq:hydro-richardson}, denoted as $\bU^{m+1} = \text{HSolver}\left(\bU^{m}\right)$;
    \STATE Evaluate ${\it diff}$ as the $L_1$ norm of the difference between successive solutions: ${\it diff} = \Vert \bU^{m+1} - \bU^{m} \Vert_1$;
    \STATE Set $m = m+1$;
    \ENDWHILE
    \STATE Extract $\rho^{m}$, $\bu^{m}$, and $\theta^{m}$ from $\bU^{m}$;
    \RETURN $\bbf^{n+1}$, constructed from $\boldsymbol{\tilde{f}}$ by replacing $\tilde{\rho}$, $\tilde{\bu}$, and $\tilde{\theta}$ with $\rho^{m}$, $\bu^{m}$, and $\theta^{m}$, respectively.
  \end{algorithmic}
\end{algorithm}

\section{Further acceleration strategies}
\label{sec:further-acceleration}

To further enhance the efficiency and robustness of the FIM method,
this section is devoted to extending the previous FIM-1 solver by
incorporating several acceleration strategies, including a
semi-implicit scheme, the Gauss-Seidel method with a cell-by-cell
sweeping strategy, and a nonlinear multigrid method. The resulting FIM
solvers are summarized at the end of the section.

\subsection{Semi-implicit scheme for moment system}
\label{sec:further-SIS}

Due to the presence of $1/\Kn$ in the collision term, the moment
system of order $M\geq 2$ exhibits increasing stiffness as $\Kn\to
0$. As a consequence, the forward Euler scheme \cref{eq:euler-scheme},
with the time step size determined by \cref{eq:dt,eq:dt-local}, might
become unstable for small $\Kn$, leading to significantly slower
convergence or even divergence of the FIM solver.
To overcome this instability, we hereafter consider
a simple semi-implicit approach for the moment system, in which the
collision term is treated implicitly.

Specifically, replacing the discrete collision term $Q(f_{\bi}^{n})$
with $Q(f_{\bi}^{n+1})$ in the forward Euler scheme
\cref{eq:euler-scheme} and rearranging terms yields
\begin{align}
  \label{eq:BGK-semi-imp}
  f_{\bi}^{n+1} - \Delta t Q(f_{\bi}^{n+1}) = f_{\bi}^{n} -\Delta
t\mA_{\bi}\left(\bbf^{n}\right),
\end{align}
where $\mA_{\bi}\left(\bbf^{n}\right)$ represents the remaining part
of $-\mR_{\bi}\left(\bbf^{n}\right)$, namely, the contribution from
the spatial discretization of the convection term
$\bxi\cdot \nabla_{\bx} f$. As described in \cref{sec:model-xdis}, this
would give rise to 
\begin{align}
  \label{eq:semi-implicit-expansion}
  f^{n+1}_{\bi, 0}\mH_{0} ^{[\bu^{n+1}_{\bi}, \theta^{n+1}_{\bi}]}
(\bxi) + \sum_{2\leq|\alpha| \leq M} \bra{1 + \Delta t \nu_{\bi}^{n+1}}
f^{n+1}_{\bi, \alpha} \mH_{\alpha} ^{[\bu^{n+1}_{\bi},
\theta^{n+1}_{\bi}]} (\bxi) = \sum_{|\alpha| \leq M} g_{\bi,\alpha}
\mH_{\alpha} ^{[\bu^{n}_{\bi}, \theta^{n}_{\bi}]} (\bxi),
\end{align}
by noting that Grad's expansion of order $M$ is employed for both
$f_{\bi}^{n}$ and $f_{\bi}^{n+1}$. Here $g_{\bi,\alpha}$ denote the
resulting coefficients evaluated from the right-hand side of
\cref{eq:BGK-semi-imp}.
Multiplying both sides of the above equation by
$\left(1,\bxi,\vert\bxi\vert^{2}\right)^{T}$, and integrating over the
whole velocity space as in \cref{eq:rho-u-theta}, we can establish
generalized relations for the macroscopic quantities and lower-order
expansion coefficients,
from which we get
\begin{align}
  \label{eq:semi-implicit-u-theta}
  f_{\bi,0}^{n+1} = g_{\bi,0}, \quad \bu_{\bi}^{n+1} = \bu_{\bi}^{n} +
\frac{1}{g_{\bi,0}}\left(g_{\bi,e_{1}},g_{\bi,e_{2}},g_{\bi,e_{3}}\right)^{T},
\quad \theta_{\bi}^{n+1} = \theta_{\bi}^{n} +
\frac{2}{3}\sum\limits_{d=1}^{3}\frac{g_{\bi,2e_{d}}}{g_{\bi,0}} -
\frac{\left\vert \bu_{\bi}^{n+1} - \bu_{\bi}^{n}\right\vert^{2}}{3}.
\end{align}
The right-hand side of \cref{eq:semi-implicit-expansion} is then
projected onto the same basis functions as those on the left-hand
side. With these re-expansion coefficients denoted by
$\tilde{g}_{\bi,\alpha}$, we finally obtain
\begin{align}
  \label{eq:semi-implicit-rest-falpha}
  f_{\bi,\alpha}^{n+1} = \frac{1}{1+\Delta t \nu_{\bi}^{n+1}}
\tilde{g}_{\bi,\alpha},\quad \text{for } \vert \alpha \vert\geq 2.
\end{align}

It can be readily observed that the above semi-implicit scheme (SIS)
incurs almost the same computational cost per time step as the forward
Euler scheme \cref{eq:euler-scheme}. However, it offers considerably
improved numerical stability, being particularly effective in mitigating
stiffness caused by small $\Kn$. This stability property is formalized
in the following theorem.

\begin{theorem}
  \label{the:semi-imp}
  The semi-implicit scheme (SIS) described in
  \cref{eq:semi-implicit-u-theta,eq:semi-implicit-rest-falpha} is
  stable, provided that the time step size $\Delta t$ satisfies the
  CFL condition \cref{eq:dt,eq:dt-local}.
\end{theorem}

\begin{proof}
  First, it is well known that the CFL condition described in
  \cref{eq:dt,eq:dt-local} ensures the numerical stability of the
  convection procedure, i.e., the computation of the right-hand side
  in \cref{eq:BGK-semi-imp}, indicating that the amplification factor
  of this step is no greater than $1$. To be specific, given a small
  perturbation $\delta f_{\bi}^{n}$ of the $n$th approximation
  $f_{\bi}^{n}$, the perturbation $\delta g_{\bi}$ resulting from
  evaluating the right-hand side of \cref{eq:semi-implicit-expansion}
  satisfies
  \begin{align*}
    \left\Vert \delta \bg \right\Vert \leq C_{1} \left\Vert \delta
    \bbf^{n} \right\Vert,
  \end{align*}
  for some constant $C_{1} \leq 1$, where $\delta\bg$ and
  $\delta \bbf^{n}$ are the collection of all $\delta g_{\bi}$ and
  $\delta f_{\bi}^{n}$, respectively.

  Second, the projection procedure, which re-expands the right-hand
  side of \cref{eq:semi-implicit-expansion} onto the basis functions
  associated with the updated parameters $\bu_{\bi}^{n+1}$ and
  $\theta_{\bi}^{n+1}$, can be carried out analytically, as
  demonstrated in \cite{Hu2020}. It follows that the computation of
  \cref{eq:semi-implicit-u-theta} and the re-expansion coefficients
  $\tilde{g}_{\bi,\alpha}$ is numerically stable. Thus, the
  corresponding perturbation $\delta \tilde{g}_{\bi}$ fulfills
  \begin{align*}
    \left\Vert \delta \tilde{g}_{\bi} \right\Vert \leq C_{2}\left\Vert
    \delta g_{\bi} \right\Vert,
  \end{align*}
  for some constant $C_{2}\leq 1$.

  In the last computational step \cref{eq:semi-implicit-rest-falpha},
  it is easy to show that each coefficient $f_{\bi,\alpha}^{n+1}$ with
  $\vert\alpha\vert\geq 2$ is scaled by a factor strictly less than
  $1$, since $\Delta t \nu_{\bi}^{n+1}>0$. Hence, the
  final perturbation $\delta f_{\bi}^{n+1}$ satisfies
  \begin{align*}
    \left\Vert \delta {f}_{\bi}^{n+1} \right\Vert < \left\Vert \delta
    \tilde{g}_{\bi} \right\Vert.
  \end{align*}
  Combining all the above inequalities, we then obtain
  \begin{align*}
    \left\Vert\delta \bbf^{n+1} \right\Vert < \left\Vert \delta
    \boldsymbol{\tilde{g}} \right\Vert \leq C_{2}\Vert \delta \bg \Vert
    \leq C_{3} \left\Vert \delta \bbf^{n} \right\Vert,
  \end{align*}
  for some constant $C_{3}\leq 1$, where $\delta \bbf^{n+1}$ and
  $\delta \boldsymbol{\tilde{g}}$ denote the collection of all
  $\delta f_{\bi}^{n+1}$ and $\delta \tilde{g}_{\bi}$, respectively.
  Therefore, the proposed SIS scheme is stable under the given CFL
  condition, with an overall amplification factor bounded by
  $1$.
\end{proof}

\subsection{Gauss-Seidel method with cell-by-cell  sweeping}
\label{sec:further-sweeping}

As revealed in \cite{Hu2016,Hu2019,Hu2023}, the convergence can be
effectively improved by modifying the forward Euler scheme
\cref{eq:euler-scheme}, which is essentially a Jacobi-type iteration,
into a Gauss-Seidel iteration with a cell-by-cell sweeping strategy.
In this Gauss-Seidel method, the solution is updated cell by cell
along some prescribed sweeping directions, with the latest
approximation on each cell directly utilized in computations for
subsequent cells. Apparently, such an approach can be extended to
all previously introduced moment and hydrodynamic solvers within the
FIM method.

In particular, for the moment system, incorporating the Gauss-Seidel
method into \cref{eq:BGK-semi-imp}, we can modify the SIS into the
following scheme
\begin{align}
  \label{eq:semi-imp-gs}
  f_{\bi}^{n+1} - \Delta t_{\bi} \mQ(f_{\bi}^{n+1}) = f_{\bi}^{n} -
  \Delta t_{\bi} \mA_{\bi}\left(\bbf^{*}\right),
\end{align}
where the local time step size $\Delta t_{\bi}$ in place of the global
one is employed, and $\bbf^{*}$ is initially set to $\bbf^{n}$. During
the sweeping process, the $\bi$th component of $\bbf^{*}$ is updated
to the new approximation $f_{\bi}^{n+1}$ immediately once it becomes
available. 

Similarly, for the hydrodynamic equations \cref{eq:hydro-eqs-vec},
applying the Gauss-Seidel method to the forward Euler scheme
\cref{eq:hydro-richardson} leads to
\begin{align}
  \label{eq:hydro-gs}
  \bU_{\bi}^{m+1} = \bU_{\bi}^m + \Delta t_{\bi} \mR^{\text{hydro}}_{\bi}(\bU^*),
\end{align}
where $\bU^{*}$ is defined analogously to $\bbf^{*}$. That is,
$\bU^{*}$ is initialized by $\bU^{m}$, and updated cell by cell as the
new approximation $\bU_{\bi}^{m+1}$ becomes available during the
sweep.

It remains to specify the sweeping directions, which play a crucial
role in the performance of the above Gauss-Seidel schemes. In this
work, we adopt a symmetric sweeping strategy for both schemes
\cref{eq:semi-imp-gs,eq:hydro-gs}, as its effectiveness and robustness
have been verified in \cite{Hu2016,Hu2019} by the resulting symmetric
Gauss-Seidel (SGS) iteration modified from the forward Euler scheme
\cref{eq:euler-scheme}. Specifically, in the one-dimensional case, the
sweeping strategy for a single SGS iteration consists of a forward
sweep followed by a backward sweep over all grid cells. In the
two-dimensional case, it involves two opposite sweeping directions, as
illustrated in \cref{fig:sweeping-2D}, where the spatial coordinates
$x_{1}$ and $x_{2}$ are replaced by $x$ and $y$, respectively, with
the convention used throughout. 
In conjunction with this symmetric sweeping strategy, the Gauss-Seidel
iteration defined in \cref{eq:semi-imp-gs} will henceforth be
abbreviated as the SISGS iteration for convenience.

\begin{figure}[!tbp]
  \centering
  \begin{tikzpicture}[scale=0.7, transform shape]
    \draw[step=0.5, help lines] (0.5,0.5) grid(4.5,4.5);
    \draw[-latex] (0.5,0) -- (0.5,5) node[left] {$y$};
    \draw[-latex] (0,0.5) -- (5,0.5) node[below] {$x$};
    \draw [fill] (0.75,0.75) circle [radius=0.05];
    \draw[-latex,red] (0.75,0.75) -- (1.75,0.75);
    \draw[-latex,blue] (0.75,0.75) -- (0.75,1.75);
    \draw[dotted,thick] (0.75,0.75) -- (0,0.75) node[left]{$
      \begin{aligned}
        & i_1 = 1,2, \ldots, N_{1}; \\ & i_2 = 1, 2, \ldots, N_{2}.
      \end{aligned}\quad (\text{D}1)$};
    \draw (5.7,2.5) node[text centered] (oloop) {Inner loop};
    \draw[dotted,blue,semithick] (0.75,1.25) -- (4.8,2.5);
    \draw[dotted,blue,semithick] (4.3,3.75) -- (4.8,2.5);
    \draw [fill] (4.25,4.25) circle [radius=0.05];
    \draw[-latex,red] (4.25,4.25) -- (3.25,4.25);
    \draw[-latex,blue] (4.25,4.25) -- (4.25,3.25);
    \draw[dotted,thick] (4.25,4.25) -- (5,4.25) node[right]{$(\text{D}2) \quad
      \begin{aligned}
        & i_1 = N_{1},\ldots,2,1; \\ & i_2 = N_{2},\ldots,2,1.
      \end{aligned}$};
    \draw (2.5,5) node[text centered] (iloop) {Outer loop};
    \draw[dotted,red,semithick] (1.25,0.75) -- (2.5, 4.75);
    \draw[dotted,red,semithick] (3.75,4.3) -- (2.5, 4.75);
  \end{tikzpicture}
  \caption{The sketch of sweeping directions $(\text{D}1)$ and
    $(\text{D}2)$ in two-dimensional case.}
  \label{fig:sweeping-2D}
\end{figure}

\subsection{Nonlinear multigrid method}
\label{sec:further-nmg}

The nonlinear multigrid (NMG) method, also known as the
full approximation scheme (FAS), is another well-established
acceleration technique for solving nonlinear problems
\cite{Brandt2011,Trottenberg2000}.
It provides a general framework, that readily extends a single
level solver into a multi-level one, where the single level solver is
utilized as the smoother, and coarse level corrections are introduced
to accelerate the convergence of the solution on the finest
level. Following this framework with appropriate restriction and
prolongation operators,
we have previously developed two types of nonlinear multi-level solvers
for the discrete steady-state problem \cref{eq:discrete-ssp}: an NMG
iteration based on spatial coarse grid corrections
\cite{Hu2014nmg,Hu2023}, where all grid levels adopt the same order
of the moment system, and an NMLM iteration based on lower-order model
corrections \cite{Hu2016,Hu2019}, where all model levels share the
same spatial grid. Both the NMG and NMLM iterations have shown to
significantly improve computational efficiency.

Motivated by this, it is natural to construct new NMG solvers by
employing one of the FIM solvers as the smoother on all grid
levels. To this end, we first clarify that the coarse grid
problem associated with the finest level problem
\cref{eq:discrete-ssp} is defined in the form
\begin{align}
  \label{eq:discrete-ssp-rhs}
  \mR_{\bi}(\bbf) = r_{\bi}(\bxi),
\end{align}
where the right-hand side $r_{\bi}(\bxi)$ is precomputed and independent of
the unknown solution $\bbf$, in accordance with the standard NMG
framework. Accordingly, the corresponding hydrodynamic equations on
the same grid level take the form
\begin{align}
  \label{eq:mac-dis-rhs}
  \mR_{\bi}^{\text{hydro}}(\bU) = r_{\bi}^{\text{hydro}}. 
\end{align}
For the purpose of consistency, the right-hand side
$r_{\bi}^{\text{hydro}}$ should not be computed in the same manner as
the computation of $r_{\bi}(\bxi)$, 
but instead evaluated from $r_{\bi}(\bxi)$ via the same procedure as
that used to derive the left-hand side. With the coarse grid problem
and the associated hydrodynamic equations defined as above, the FIM
alternating iteration can be conveniently applied
with only minimal adjustments, compared to its application described
in previous subsections for the original problem
\cref{eq:discrete-ssp}. Incorporating the remaining components of the NMG method
introduced in \cite{Hu2014nmg,Hu2023}, a new NMG solver can thus be
obtained.

\begin{remark}
  It is noteworthy that, in some sense, the NMLM solver can be
  viewed as a multigrid method in velocity space \cite{Hu2019}, and
  the FIM solver can be interpreted as a two-level NMLM
  solver. Consequently, the NMG solver that employs the FIM solver as
  the smoother constitutes an integrated multi-level method, which
  accelerates convergence by simultaneously accounting for coarse
  level corrections in both spatial and velocity spaces.
  However, it is also
  observed that replacing the smoother with the original two-level
  NMLM solver does not result in a solver as robust as the present
  integrated multi-level method.
\end{remark}

\subsection{FIM solver family}

With the acceleration strategies introduced above, the FIM method
gives rise to a family of solvers, which share a common alternating
iteration framework, but differ in the specific numerical schemes used
for the moment system and the hydrodynamic equations. In addition to
the FIM-1 solver, other FIM solvers investigated in our numerical
experiments include the FIM-2 and FIM-3 variants. The corresponding
schemes adopted in these FIM solvers are summarized in
\cref{tab:fim-solvers}.

Moreover, several further enhanced solvers can be constructed within
the NMG framework by employing one of the FIM solvers as the
smoother. For brevity, the resulting FIM-based NMG solver is referred
to as the NMG solver in the following numerical experiments, as only
the version using the FIM-3 solver as the smoother is taken into
account.

\begin{table}[!tbp]
  \tabcolsep=0.2cm
  \centering\small
  \caption{FIM solver setup.}
  \label{tab:fim-solvers}
  \begin{tabular}{ccc}
    \toprule
    \multirow{2}*{\quad FIM solver\quad} & \multicolumn{2}{c}{Schemes} \\
    \cmidrule(r){2-3}
     & High-order moment system & Hydrodynamic equations \\
    \midrule
    FIM-1 & Forward Euler \cref{eq:euler-scheme} & Forward Euler \cref{eq:hydro-richardson} \\
    FIM-2 & SIS \cref{eq:semi-implicit-u-theta,eq:semi-implicit-rest-falpha} & Forward Euler \cref{eq:hydro-richardson} \\
    FIM-3 & SISGS \cref{eq:semi-imp-gs} & SGS \cref{eq:hydro-gs}  \\
    \bottomrule
  \end{tabular}
\end{table}


\section{Numerical experiments}
\label{sec:experiments}

To illustrate the performance of the proposed solvers, we carry out
three numerical experiments, including the planar Couette flow, the
shock structure, and the lid-driven cavity flow. In all experiments,
the test gas is argon, with a molecular mass of
$m_{*} = \qty{6.63e-26}{\kg}$. Unless otherwise specified, the
following settings are adopted. The CFL number controlling the time
step size, the tolerance for convergence, and the number of steps for
the moment solver in each FIM alternating iteration, are set to
$\text{CFL} = 0.8$, $\mathit{Tol} = \num{e-8}$, and $\gamma_{1} = 1$,
respectively. For the associated NMG solver, a $V$-cycle is employed
with pre-, post- and coarsest-smoothing steps chosen as
$s_{1}=s_{2}=2$, and $s_{3} = 4$, respectively, to achieve a good
balance between convergence rate and computational
efficiency. Moreover, the total number of grid levels in the NMG
solver is selected such that the coarsest grid consists of $8$ cells
in the one-dimensional case and $8\times 8$ cells in the
two-dimensional case, as adopted in \cite{Hu2014nmg, Hu2023}.

Additionally, in order to ensure consistency with the steady-state
solution obtained by time-integration methods, the correction
mentioned in \cite{Hu2014nmg,Hu2023} is also applied at each FIM
alternating iteration or NMG iteration, except in the shock structure
experiment, where the prescribed boundary conditions are sufficient to
guarantee a unique steady-state solution.

\subsection{Planar Couette flow}
\label{sec:couette}

The first example is the planar Couette flow, which is a widely
used one-dimensional benchmark problem
\cite{Hu2014nmg,Zhu2016,Zhu2021}. For convenience, we consider the
same dimensionless setup and parameters as those used in
\cite{Cai2013microflow,Hu2014nmg}. Particularly, the gas is confined between
two infinite parallel plates separated by a distance $L = 1$. Both
plates are maintained at the temperature of $\theta^{W}=1$, and move in 
opposite directions along the plate with a relative velocity
$\bu^{W} = (0, 1.2577, 0)^{T}$. The collision frequency $\nu$ is
defined as
\begin{align}
  \label{eq:nu-couette}
  \nu = \sqrt{\frac{\pi}{2}} \frac{1}{\Kn} \rho \theta^{1 - \omega},
\end{align}
where $\omega=0.81$ is the viscosity index. In all tests, the gas is
initially assumed to be uniformly distributed by the global Maxwellian
with the macroscopic quantities specified as
\begin{align}
  \label{eq:inital_M}
  \rho_0 = 1, \quad \bu_0 = 0, \quad \theta_0 = 1.
\end{align}
Driven by the motion of the plates, it then gradually evolves
toward a steady state.

\subsubsection{Solution validation}
\label{sec:couette-validation}

Three different Knudsen numbers in the near-continuum flow regime, namely
$\Kn=0.1$, $0.01$, and $0.001$, are investigated. For comparison and
validation, reference solutions are obtained from the dugksFoam solver
introduced in \cite{Zhu2017foam} for $\Kn=0.1$ and $0.01$, and from
the NSF equations for $\Kn=0.001$. As mentioned in
previous sections and observed in
\cite{Cai2013microflow,Hu2014nmg}, the order $M$ of up to $9$ or $10$
may be required for the moment system to generate satisfactory results
at $\Kn=0.1$, while a smaller value of $M$ would be sufficient for a
smaller $\Kn$. Accordingly, we just take $M=9,10$ for $\Kn=0.1$,
$M=7,8$ for $\Kn=0.01$, and $M=5,6$ for $\Kn=0.001$, respectively, for
simplicity. Numerical solutions for the density $\rho$, temperature
$\theta$, vertical velocity $u_{2}$, and heat flux $q_{1}$ on the grid
with $N_{1}=1024$ are presented in \cref{fig:couette_solutions}. It
can be seen that all of our results agree well with the corresponding
references, which confirm the validity of the high-order
moment system.

\begin{figure}[!tbp]
  \centering
  \subfloat[Density, $\rho$]{
    \label{fig:couette_rho}
    \includegraphics[width=0.39\textwidth]{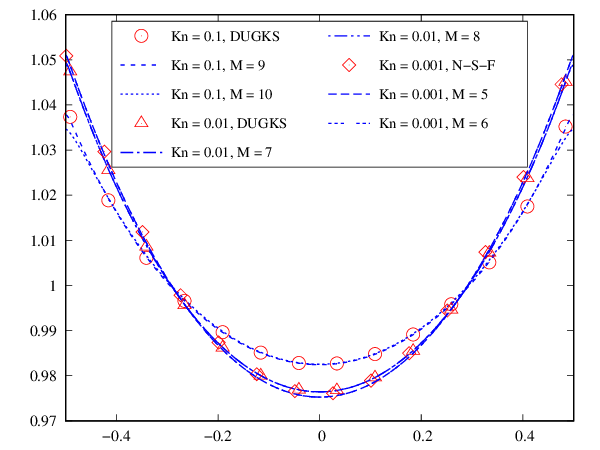}
  }\quad
  \subfloat[Temperature, $\theta$]{
    \label{fig:couette_T}
    \includegraphics[width=0.39\textwidth]{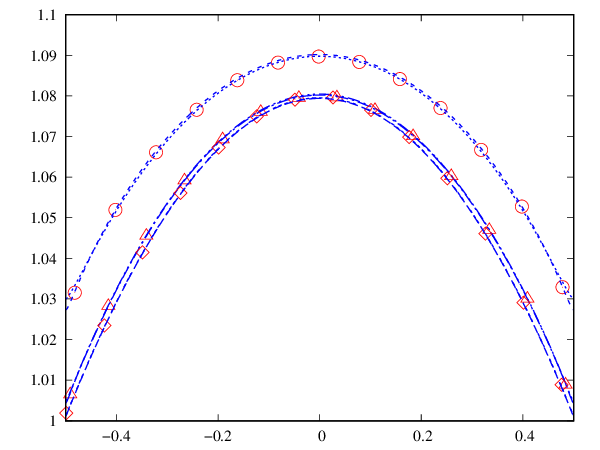}
  } \\
  \subfloat[Velocity, $u_2$]{
    \label{fig:couette_v}
    \includegraphics[width=0.39\textwidth]{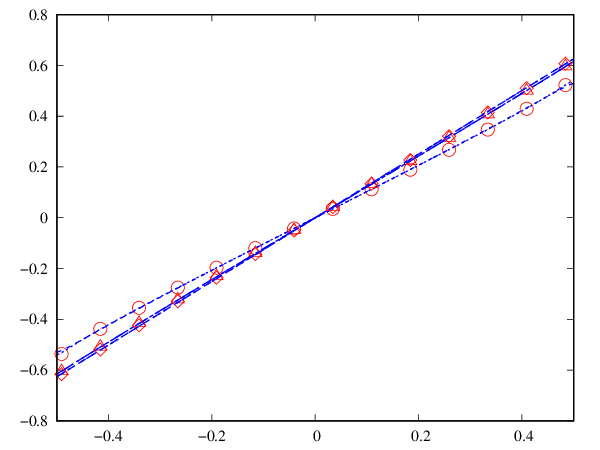}
  }\quad
  \subfloat[Heat flux, $q_{1}$]{
    \label{fig:couette_q}
    \includegraphics[width=0.39\textwidth]{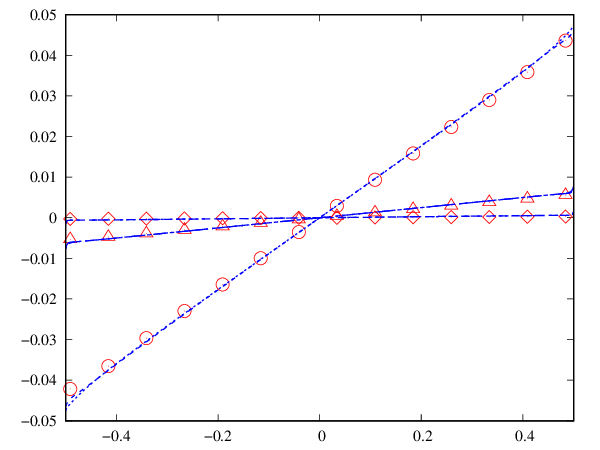}
  } 
  \caption{Numerical solutions of Couette flow at different $\Kn$
    on the uniform grid with $N_1 = 1024$.}
  \label{fig:couette_solutions}
\end{figure}

\subsubsection{Performance of basic moment solvers}
\label{sec:couette-implicit}

The robustness and convergence behavior of the SIS and SISGS
iterations, as well as the forward Euler scheme
\cref{eq:euler-scheme}, are first examined, 
by performing them directly as basic moment solvers. The
total number of iterations and the corresponding wall-clock time,
taken by these solvers to reach the steady state on the grid with
$N_{1} = 512$, are listed in \cref{tab:couette-time-SIS}.

As expected and consistent with algorithmic characteristics, the
convergence of all three methods deteriorates significantly as $\Kn$
decreases toward the continuum flow regime, leading to a substantial
increase in the total number of iterations. Even worse, the forward
Euler scheme fails to converge at $\Kn=0.001$.
Although reducing the CFL number could eventually enable the forward
Euler scheme to converge, the SIS remains stable and convergent with
the default $\text{CFL}=0.8$. This clearly shows the superior
robustness of the SIS over the forward Euler scheme at small $\Kn$.
In contrast, for $\Kn = 0.1$ and $0.01$ where both the forward Euler
scheme and SIS converge, these two methods require approximately the
same number of iterations and wall-clock time, indicating comparable
performance in such cases.

Moreover, benefiting from the symmetric cell-by-cell sweeping
strategy, the SISGS iteration greatly accelerates
convergence. Specifically, it reduces the total number of iterations
to less than one-third and the wall-clock time by more than half,
compared to the SIS in all tests. Consequently, the SISGS iteration
achieves a considerable improvement in efficiency over both the
forward Euler scheme and SIS. However, all three methods exhibit
similar scaling behavior with respect to the grid number $N_{1}$. As
shown in \cref{tab:couette-basic-solver-grids}, when $N_{1}$ is
doubled, the total number of iterations for $\Kn=0.1$ roughly doubles,
accompanied by a nearly fourfold increase in wall-clock
time. Meanwhile, this scaling behavior also deteriorates as $\Kn$
decreases, resulting in the total number of iterations growing by a
factor exceeding $3$ for both SIS and SISGS methods at $\Kn=0.001$.

Additionally, despite the notable increase in the number of iterations
as $\Kn$ decreases, the overall growth in wall-clock time, as shown in
\cref{tab:couette-time-SIS}, remains relatively moderate. This is
mainly because lower orders of the moment system are utilized for smaller $\Kn$,
thereby reducing the computational cost per iteration.

\begin{table}[!tbp]
  \tabcolsep=0.2cm
  \centering\small
  \caption{Performance of basic moment solvers for Couette
    flow with $N_1 = 512$. Euler: forward Euler
    scheme; 
    SIS: semi-implicit scheme; SISGS: SIS-based symmetric Gauss-Seidel
    iteration; ``--'': failed to converge.}
  \label{tab:couette-time-SIS}
  \begin{tabular}{cccccccc}
    \toprule
    \multirow{2}*{$\Kn$} & \multirow{2}*{$M$} & \multicolumn{3}{c}
     {Iterations} & \multicolumn{3}{c}{Wall-clock time $(\unit{\s})$} \\
    \cmidrule(r){3-5} \cmidrule(r){6-8}
                         &  & Euler & SIS & SISGS & Euler  & SIS & SISGS \\
    \midrule
    $\num{0.1}$ 
                         & 10 & 62074 & 62068 & 18042 & 13469.29 & 13580.11 & 5965.08 \\
                         & 9 & 58413 & 58407 & 16912 & 9987.99 & 10078.74 & 4282.54 \\
    $\num{0.01}$ 
                         & 8 & 217947 & 217941 & 62472 & 27109.30 & 27178.26 & 11834.03\\
                         & 7 & 205063 & 205057 & 58336 & 18229.03 & 18594.80 & 7918.06 \\
    $\num{0.001}$ 
                         & 6 & -- & 450617 & 128418 & -- & 28260.00 & 12192.43 \\
                         & 5 & -- & 435694 & 122754 &  -- & 18604.39 & 7826.48 \\
    \bottomrule
  \end{tabular}
\end{table}

\begin{table}[!tbp]
  \tabcolsep=0.2cm
  \centering\small
  \caption{Iterations of basic moment solvers for Couette flow on
    grids with $N_1 = 128$ and $256$. The Ratio columns show the
    associated ratios between successive grids, with data for $N_{1}=512$
    taken from \cref{tab:couette-time-SIS}.
  }
  \label{tab:couette-basic-solver-grids}
  \begin{tabular}{cccccccc}
    \toprule
     & \multirow{2}*{$N_{1}$} & \multicolumn{2}{c}{$\Kn=0.1$, $M=9$} & \multicolumn{2}{c}{$\Kn=0.01$, $M=7$} & \multicolumn{2}{c}{$\Kn=0.001$, $M=5$}\\
    \cmidrule(r){3-4} \cmidrule(r){5-6} \cmidrule(r){7-8}
                         &  & Iterations & Ratio & Iterations & Ratio  & Iterations & Ratio \\
    \midrule
    Euler             & 128 & 12679 & 2.19 & 28150  & 2.83 &  & \\
                      & 256 & 27718 & 2.11 & 79526  & 2.58 &  & \\
    SIS               & 128 & 12675 & 2.19 & 28147  & 2.83 & 36847  & 3.49 \\
                      & 256 & 27713 & 2.11 & 79521  & 2.58 & 128748 & 3.38 \\
    SISGS             & 128 & 3706  & 2.17 & 8061   & 2.81 & 10619  & 3.45 \\
                      & 256 & 8043  & 2.10 & 22671  & 2.57 & 36668  & 3.35 \\
    \bottomrule
  \end{tabular}
\end{table}

\subsubsection{Performance of FIM and NMG solvers}
\label{sec:couette-fim}

Now let us turn to investigating the performance of the FIM
solvers. As stated in \cref{sec:mac-acc}, when $\Kn$ is small, the
hydrodynamic equations are able to predict
the primary macroscopic quantities $\rho$, $\bu$, and $\theta$
reasonably well, while the remaining moments, including $\bsigma$ and
$\bq$, tend to evolve more slowly. Supported by this observation and
several preliminary simulations, we set the parameter
$\gamma_{2}$ in each FIM alternating iteration to $40$, $300$, and
$600$ for $\Kn=0.1$, $0.01$, and $0.001$, respectively. This allows
the primary macroscopic quantities to evolve sufficiently before
updating the other moments, which in turn improves computational
efficiency by accelerating convergence.

The total number of iterations and the associated wall-clock time for
the FIM-1, FIM-2, and FIM-3 solvers on the grid with $N_{1}=512$ are
presented in \cref{tab:couette-FIM}.
Compared to the results of the corresponding basic moment solver
shown in \cref{tab:couette-time-SIS}, the three FIM solvers
drastically reduce the number of iterations across all cases,
indicating a remarkable improvement in convergence rate, especially
when $\Kn$ is small. In more detail, at $\Kn=0.1$, the total number of
iterations for each FIM solver drops by exceeding one order of
magnitude, requiring only a few thousand iterations to achieve the
steady state. As $\Kn$ decreases to $0.01$ and $0.001$, the total
number of iterations does not increase significantly, unlike what is
observed for the basic moment solvers. Instead, it continues to
decline, falling to less than $1/5$ of the number at $\Kn=0.1$.
For instance, the FIM-3 solver converges within no more than $250$
iterations. Consequently, the computational cost is dramatically
reduced, as reflected in the last column of \cref{tab:couette-FIM}, which lists
the wall-clock time ratio of the FIM-3 solver to the basic SISGS
solver. More specifically, when $\Kn=0.1$, the FIM-1, FIM-2, and FIM-3 solvers require
less than $8\%$ of the wall-clock time relative to their corresponding
basic moment solvers, namely, the forward Euler scheme, the SIS, and
the SISGS iteration, respectively. By contrast, these wall-clock time
ratios drop further to below $0.8\%$ at $\Kn=0.01$, and then rebound
slightly to reach up to $1.2\%$ at $\Kn=0.001$. This rebound should be
mainly attributed to the use of a moderately large value of
$\gamma_{2}$, which has not yet been carefully optimized, leading to
both a higher per-iteration cost and a limited increase in the number
of iterations, in comparison with the case at $\Kn=0.01$.

\begin{table}[!tbp]
  \tabcolsep=0.2cm
  \centering\small
  \caption{Performance of FIM solvers for Couette flow
    with $N_{1}=512$ and $\gamma_{2}=40$, $300$, $600$ for $\Kn=0.1$,
    $0.01$, $0.001$, respectively.
    $T_{r}$: wall-clock time ratio of FIM-3 to SISGS; ``--'':
    unavailable.}
  \label{tab:couette-FIM}
  \begin{tabular}{cccccccccc}
    \toprule
    \multirow{2}*{$\Kn$} & \multirow{2}*{$M$} &
    \multicolumn{3}{c}{Iterations} & \multicolumn{3}{c}{Wall-clock time
    $(\unit{\s})$} & \multirow{2}*{$T_r$} \\
    \cmidrule(r){3-5} \cmidrule(r){6-8}
                         &  & FIM-1 & FIM-2 & FIM-3 & FIM-1 & FIM-2 & FIM-3 &  \\
    \midrule
    $ \num{0.1}$ 
                         & 10 & 4454 & 4463 & 1247 & 1023.89 & 924.81 & 441.86 & 7.4 \% \\
                         & 9  & 4166 & 4176 & 1167 & 737.20 & 677.05 & 321.27 & 7.5 \% \\
    $\num{0.01}$ 
                         & 8 & 557 & 660 & 203 & 123.12 & 142.42 & 74.33 & 0.63 \% \\
                         & 7  & 525 & 613 & 194 & 96.82 & 108.63 & 60.38 & 0.76 \% \\
    $\num{0.001}$ 
                         & 6 & - & 824 & 249 & - & 205.74 & 104.24 & 0.85 \% \\
                         & 5  & - & 799 & 242 & - & 185.03 & 93.79 & 1.2 \% \\
    \bottomrule
  \end{tabular}
\end{table}

To explore the behavior of the FIM solvers throughout the iteration
process, we illustrate their convergence histories in terms of
wall-clock time, alongside those of the NMG and basic SISGS
solvers, in \cref{fig:couette-con-his-fim}. The figure reveals not
only the high efficiency of the FIM solvers, but also their generally
steady residual decay, except for the FIM-1 and FIM-2 solvers at
$\Kn=0.01$, where some minor oscillations appear, attributable to
instability induced by unoptimized choices of both the CFL number and
$\gamma_{2}$. This increased instability at smaller $\Kn$ ultimately
causes the FIM-1 solver to fail to converge at $\Kn=0.001$. However,
owing to the implicit treatment of the collision term, both the FIM-2
and FIM-3 solvers continue to converge smoothly at $\Kn=0.001$,
demonstrating their superior robustness.

\begin{figure}[!tbp]
  \centering
  \subfloat[$\Kn=0.1$, $M=9$]{
    \includegraphics[width=0.31\textwidth]{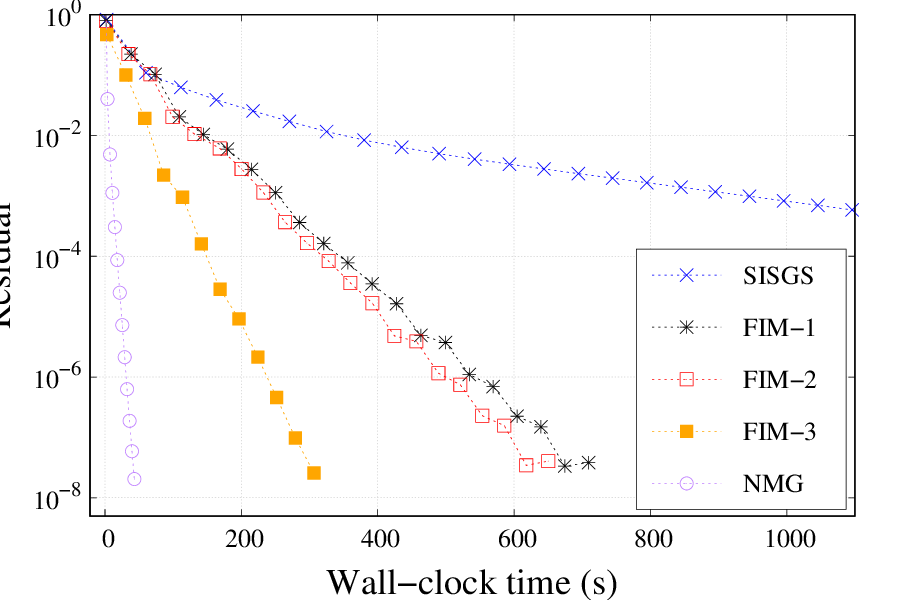}
  }\hfill
  \subfloat[$\Kn=0.01$, $M=7$]{
    \includegraphics[width=0.31\textwidth]{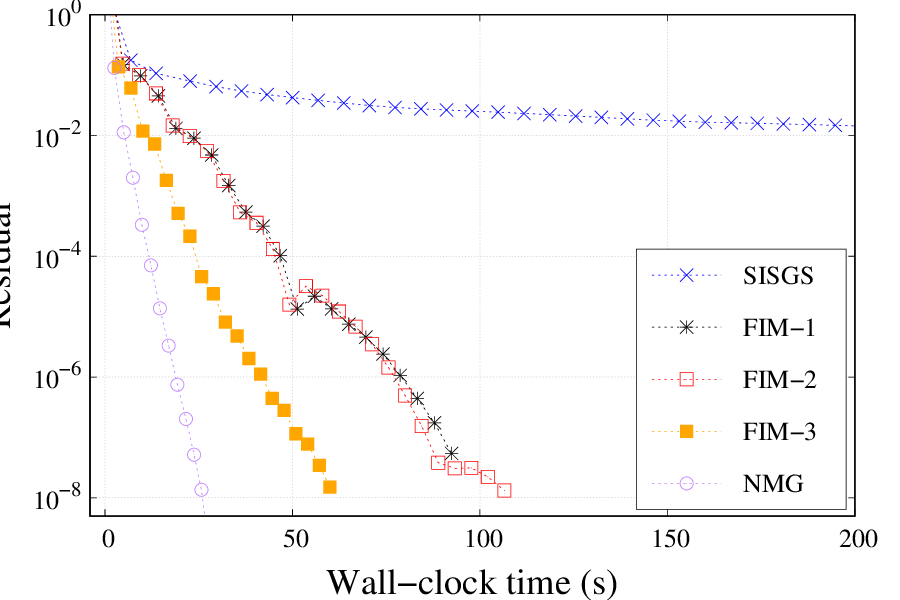}
  }\hfill
  \subfloat[$\Kn=0.001$, $M=5$]{
    \includegraphics[width=0.31\textwidth]{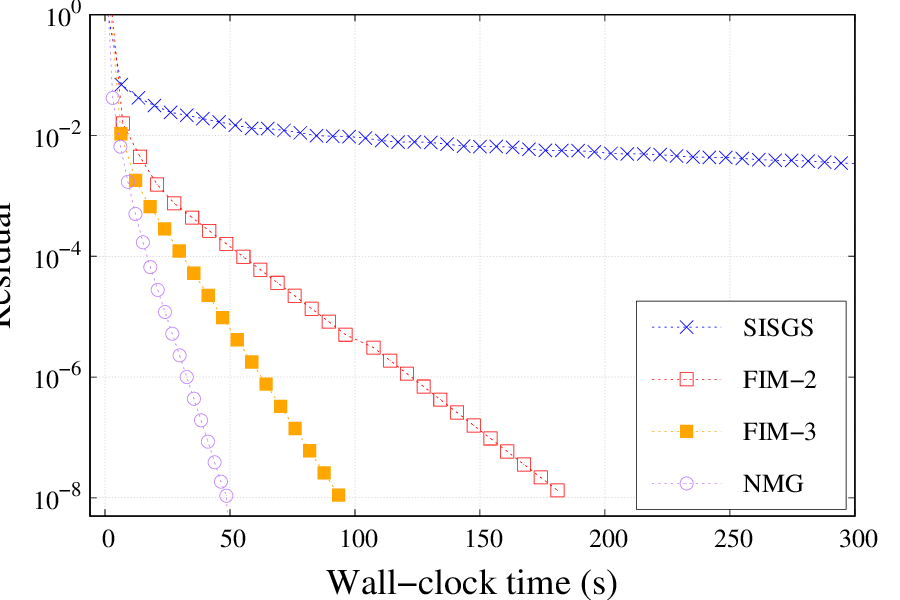}
  }
  \caption{Convergence histories of various solvers for Couette flow on
    the uniform grid with $N_1 = 512$.}
  \label{fig:couette-con-his-fim}
\end{figure}

The scaling behavior of the FIM solvers under grid refinement is then
examined. Since the three FIM solvers exhibit nearly identical trends
with respect to the grid number $N_{1}$, we present only the total
number of iterations for the FIM-3 solver on two additional grids with
$N_{1}=128$ and $256$ in \cref{tab:couette-FIM-grids}, along with the
associated iteration growth ratios between successive grid levels. It
is observed that these growth ratios remain close to $2$ at $\Kn=0.1$,
and increase to around $3.4$ at $\Kn=0.001$. By comparing these values
with those listed in \cref{tab:couette-basic-solver-grids}, we find
that the FIM solvers share almost the same grid-dependent scaling
behavior as the basic moment solvers across all three Knudsen
numbers. In particular, although the growth ratio deteriorates as
$\Kn$ decreases, the relative acceleration of the FIM solvers over the
basic moment solvers is effectively preserved with negligible
dependence on the grid, suggesting that their computational advantage
is retained even as the grid is refined.

\begin{table}[!tbp]
  \tabcolsep=0.3cm
  \centering\small
  \caption{Iterations of the FIM-3 solver for Couette flow on
    various grids with $\gamma_{2}=40$, $300$, $600$ for $\Kn=0.1$,
    $0.01$, $0.001$, respectively. The Ratio columns show the
    associated ratios between successive grids, with data for $N_{1}=512$
    taken from \cref{tab:couette-FIM}.
  }
  \label{tab:couette-FIM-grids}
  \begin{tabular}{ccccccc}
    \toprule
    \multirow{2}*{$N_{1}$} & \multicolumn{2}{c}{$\Kn=0.1$, $M=9$} & \multicolumn{2}{c}{$\Kn=0.01$, $M=7$} & \multicolumn{2}{c}{$\Kn=0.001$, $M=5$}\\
    \cmidrule(r){2-3} \cmidrule(r){4-5} \cmidrule(r){6-7}
             & Iterations & Ratio & Iterations & Ratio  & Iterations & Ratio \\
    \midrule
         128 & 294  & 1.99 & 32   & 2.40 & 21  & 3.43 \\
         256 & 585  & 1.99 & 77   & 2.51 & 72  & 3.36 \\
    \bottomrule
  \end{tabular}
\end{table}

To further investigate the performance of the FIM solvers when used
within the NMG framework, we take the NMG solver employing the FIM-3
solver for the smoother as a representative example. The total number
of iterations and the corresponding wall-clock time, required by the
NMG solver to reach the steady state on three grids with $N_{1}=512$,
$1024$, and $2048$, are reported in \cref{tab:couette-nmg}, of which
the convergence histories, showing smooth and efficient residual decay
for $N_1 = 512$, are illustrated in \cref{fig:couette-con-his-fim}. In
all cases, convergence is achieved within a few dozen iterations,
leading to an even more significant reduction in computational cost
than the FIM-3 solver. Specifically, on the grid with $N_{1}=512$, the
NMG solver consumes only about $14\%$ of the wall-clock time required
by the FIM-3 solver at $\Kn=0.1$. While the speedup becomes less
pronounced at smaller $\Kn$, the NMG solver still achieves a
notable reduction in wall-clock time, with the ratio dropping to
roughly $50\%$ at both $\Kn=0.01$ and $0.001$. This modest degradation
in acceleration at smaller $\Kn$ is acceptable, considering that the
FIM-3 solver itself already significantly outperforms the basic SISGS
solver in these regimes. Actually, when comparing the NMG solver
directly against the basic SISGS solver, the overall speedup becomes
even more striking for all three $\Kn$: on the grid with
$N_{1} = 512$, the NMG solver requires just about $1\%$ of the
wall-clock time consumed by the basic SISGS solver at $\Kn = 0.1$, and
this fraction drops even lower to approximately $0.35\%$ at
$\Kn = 0.01$ and $0.64\%$ at $\Kn = 0.001$.

As the grid is refined, the total number of iterations taken by the
NMG solver increases moderately. To better illustrate this
grid-dependent scaling behavior, \cref{fig:couette-ite-N1} plots the
iteration growth trends with respect to the grid number $N_{1}$ for
both the NMG and FIM-3 solvers. Therein the numbers of iterations for
the FIM-3 solver are normalized by $N_{t}/5$, where $N_{t}$ denotes
the number of iterations on the grid with $N_{1}=128$, as listed in
the first row of \cref{tab:couette-FIM-grids}. While the total number
of iterations for the FIM-3 solver grows approximately linearly at
$\Kn=0.1$ and even superlinearly at smaller $\Kn$, the growth
exhibited by the NMG solver remains much lower and is further
restrained on finer grids across all three cases. This contrast
highlights the enhanced computational advantage of the NMG solver
under grid refinement, with benefit becoming particularly pronounced
at smaller Knudsen numbers.

As a supplement, we fix the order of the moment system at $M=5$ and display the
wall-clock time consumed by the FIM-2, FIM-3, and NMG solvers, along
with the basic SISGS solver, on the grid with $N_{1}=512$ for varying
Knudsen numbers in \cref{fig:couette-time-kn}. 
It is reconfirmed that the basic SISGS solver incurs a considerable
increase in computational cost with decreasing $\Kn$, due to
deteriorating convergence and the constant per-iteration cost. In
contrast, with the help of the proposed FIM alternating iteration, the
resulting solvers effectively suppress the computational cost across
all tested $\Kn$, leading to significantly improved efficiency.
Although the acceleration effect of the FIM solvers becomes less
noticeable at $\Kn=0.1$, their integration into the NMG framework still ensures
overall efficiency comparable to that achieved at smaller
$\Kn$. Meanwhile, a slight rebound in wall-clock time is observed at
$\Kn=0.001$ for both the FIM and NMG solvers, suggesting that further
optimization is necessary and potentially effective in this regime.

Finally, since higher-order moment systems contain more
equations, basic moment solvers for them typically
suffer from higher computational cost per iteration, often leading to
a much longer overall runtime.
However, the previous results reveal that the computational cost of
FIM solvers can be effectively controlled even for a relatively
high order $M$ when $\Kn$ is small. This indicates that in the
near-continuum flow regime, the FIM and FIM-based NMG solvers permit the
use of moderately large $M$ without incurring a substantial increase
in wall-clock time, thereby offering greater flexibility in choosing
$M$ for practical applications.

\begin{table}[!tbp]
  \tabcolsep=0.2cm
  \centering\small
  \caption{Performance of the NMG solver for Couette flow,
    using FIM-3 as the smoother, on grids with
    $N_{1}=512$, $1024$, and $2048$.
    $T_{r}$: wall-clock time ratio of NMG to FIM-3 with $N_1 = 512$.}
  \label{tab:couette-nmg}
  \begin{tabular}{ccccccccc}
    \toprule
    \multirow{2}*{$\Kn$} & \multirow{2}*{$M$} & \multicolumn{3}{c}
    {Iterations} & \multicolumn{3}{c}{Wall-clock time $(\unit{\s})$} &
    \multirow{2}*{$T_r$} \\
    \cmidrule(r){3-5} \cmidrule(r){6-8}
                         & & 512 & 1024 & 2048 & 512 & 1024 & 2048 & \\
    \midrule
    $\num{0.1}$ 
                         & 10 & 27 & 40 & 61 & 61.06 & 179.55 & 546.82 & 13.8 \% \\
                         & 9 & 25 & 34 & 44 & 44.72 & 120.23 & 311.46 & 13.9 \% \\
    $\num{0.01}$ 
                         & 8 & 13 & 17 & 29 & 34.05 & 82.51 & 274.18 & 45.8 \% \\ 
                         & 7 & 12 & 16 & 24 & 27.14 & 65.79 & 197.00 & 44.9 \% \\
    $\num{0.001}$ 
                         & 6 & 18 & 28 & 42 & 54.69 & 187.32 & 566.01 & 52.5 \% \\
                         & 5 & 18 & 27 & 39 & 49.75 & 167.03 & 490.55 & 53.0 \% \\
    \bottomrule
  \end{tabular}
\end{table}

\begin{figure}[!tbp]
  \centering
  \subfloat[Iterations vs. $N_{1}$]{
    \label{fig:couette-ite-N1}
    \includegraphics[width=0.39\textwidth]{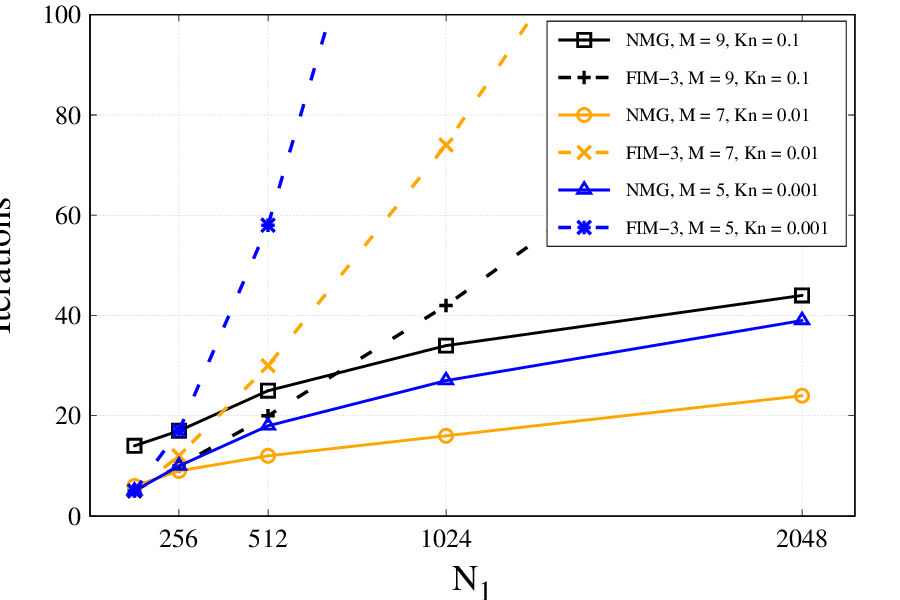}
  }\qquad 
  \subfloat[$M=5$, $N_{1}=512$]{
    \label{fig:couette-time-kn}
    \includegraphics[width=0.39\textwidth]{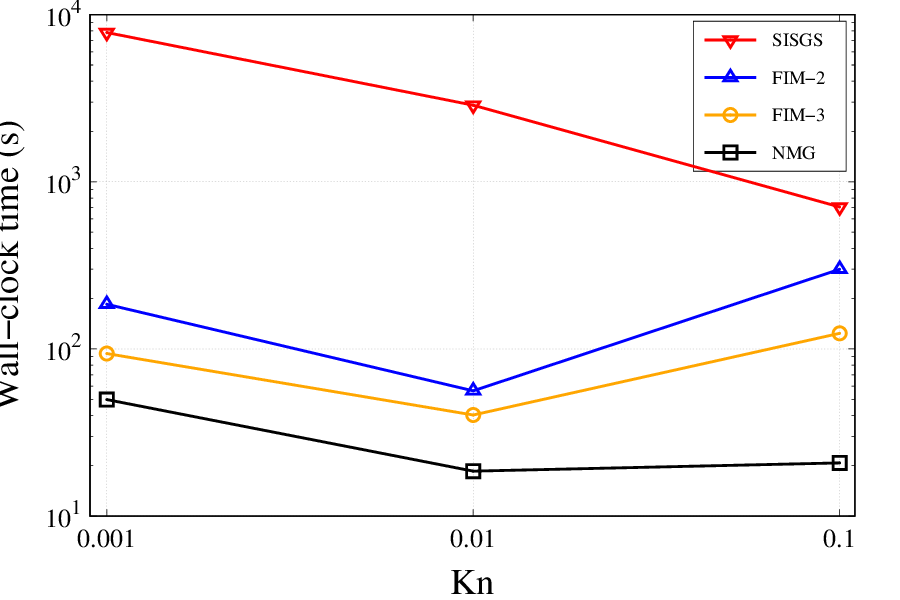}
  }
  \caption{Performance of various solvers for Couette flow. The number
    of iterations for the FIM-3 solver is normalized by ${N}_{t}/5$, with
    $N_{t}$ corresponding to the number of iterations for $N_{1}=128$, as
    given in \cref{tab:couette-FIM-grids}.
  }
  \label{fig:couette-nmg}
\end{figure}

\subsection{Shock structure}
\label{sec:shock_structure}

The plane wave shock structure is another classical benchmark problem
to test numerical methods for the Boltzmann equation \cite{Xu2011,
  Cai2012high, Cai2021}. The steady shock structure of Mach number
$\Ma$ can be obtained by evolving a one-dimensional Riemann problem,
where the initial condition is assumed to be a Maxwellian with
the macroscopic quantities given on the left side $x<0$ by
\begin{align}
  \label{eq:left-state}
  \rho_l & = 1.0, \quad \bu_l  = \left( \sqrt{{5}/{3}} \Ma, 0, 0
           \right)^{T}, \quad \theta_l  = 1.0,
\end{align}
and on the right side $x>0$ by
\begin{align}
  \label{eq:right-state}
  \rho_r & = \frac{4 \Ma^2}{\Ma^2 +3}, \quad \bu_r  = \left(
           \sqrt{\frac{5}{3}} \frac{\Ma^2 + 3}{4 \Ma}, 0, 0 \right)^{T}, \quad
           \theta_r  = \frac{5 \Ma^2 - 1}{4 \rho_r}.
\end{align}
Numerically, the computational domain is
taken as $[-1.5, 1.5]$, with boundary conditions imposed on each side
according to the corresponding initial Maxwellian. To facilitate
comparison with the experimental data in \cite{Alsmeyer1976}, the
collision frequency for the variable hard sphere model (see, e.g.,
\cite{Bird1994}) is adopted as
\begin{align}
  \label{eq:nu-VHS}
  \nu = \sqrt{\frac{2}{\pi}} \frac{ (5 - 2 \omega) (7 - 2 \omega)}{15 \Kn }
  \rho \theta^{1-\omega}.
\end{align}
Consistent with the setup in \cite{Alsmeyer1976, Cai2012high}, we set
the viscosity index to $\omega = 0.72$ and the Knudsen number to
$\Kn=0.1$. Below two Mach numbers $\Ma=1.4$ and $1.55$ are taken into
account. For simplicity, the order of the moment system is fixed at $M=7$ and the
convergence tolerance is chosen as $\mathit{Tol} = \num{5e-5}$.

The shock profiles computed on the grid with $N_{1}=1024$ are shown in
\cref{fig:shock-sol}, with the macroscopic quantities normalized as
\begin{align}
  \label{eq:normalize}
  \overline{\rho} = \frac{\rho - \rho_l}{\rho_r - \rho_l}, \quad
  \overline{u}_1 = \frac{u_1 - u_{1,r}}{u_{1,l} - u_{1,r}},\quad
  \overline{\theta} = \frac{\theta - \theta_l}{\theta_r - \theta_l}.
\end{align}
As stated in \cite{Cai2021}, the moment system of order $M=7$ produces
smooth shock structures for both Mach numbers, in good agreement with
the experimental data.

\begin{figure}[!tbp]
  \centering
  \subfloat[$\Ma = 1.4$]{
    \label{fig:shock-Ma-1.4-rho-u-T}
    \scalebox{0.5}{\input{shockStructure_Ma_1.4_BGK_rho.tex}}
  }\quad 
  \subfloat[$\Ma = 1.55$]{
    \label{fig:shock-Ma-1.55}
    \scalebox{0.5}{\input{shockStructure_Ma_1.55_BGK_rho.tex}}
  }
  \caption{The shock structure for the moment system of order $M=7$ on
    the uniform grid with $N_1 = 1024$.}
  \label{fig:shock-sol}
\end{figure}

The total number of iterations, along with the 
wall-clock time, consumed by the basic moment solvers and the FIM
solvers on the grid with $N_{1}=1024$, are presented in
\cref{tab:shock-SIS-FIM}, where the parameter $\gamma_{2}$ in the FIM
solvers is chosen as $5$. It is observed that all of these solvers
require somewhat more iterations at smaller $\Ma$.
For both Mach numbers, the forward Euler scheme and SIS exhibit nearly
identical performance, converging in approximately the same number of
iterations and consuming comparable wall-clock time as in the Couette
flow tests. In contrast, the SISGS method greatly improves
efficiency, with the total number of iterations reduced to roughly
one-third and the wall-clock time to about half.
A similar scaling trend is observed for the FIM solvers, which
drastically reduce both the total number of iterations and wall-clock
time to nearly $1/5$ of those incurred by the corresponding basic
moment solvers, resulting in a further remarkable gain in efficiency,
although the speedup appears slightly less impressive than that
achieved in the Couette flow tests. While additional tests reveal
that increasing $\gamma_{2}$ generally leads to even faster
convergence for the shock structure problem, albeit with a minor
increase in instability risk, the current choice of $\gamma_{2}=5$
already yields highly efficient performance, especially when the FIM
solver is incorporated into the NMG framework.

In \cref{tab:shock-nmg}, the results of the NMG solver on four grids
with $N_{1}$ ranging from $1024$ to $8192$ are presented. Therein the
FIM-3 solver is employed as the smoother, and the total number of grid
levels is determined by fixing the coarsest grid to consist of $128$
cells. Across all tests, convergence is achieved within $24$
iterations, demonstrating excellent robustness and a further
significant reduction in computational cost. Specifically, the
wall-clock time on the grid with $N_{1} = 1024$ is only about $10 \%$
of that required by the FIM-3 solver on the same grid. Combined with
the time ratio listed in \cref{tab:shock-SIS-FIM}, this indicates that
the overall computational cost of the NMG solver is approximately
$2.16\%$ compared to that of the basic SISGS solver. Moreover, since
the total number of iterations remains nearly independent of grid
resolution, the efficiency gain becomes increasingly pronounced as the
grid is refined.

Additionally, \cref{fig:shock-efficiency} displays the convergence
histories of the FIM, NMG and basic SISGS solvers with respect to
wall-clock time, along with a comparison of their total wall-clock
time. All FIM and NMG solvers exhibit steady and efficient residual
decay, although a slight slowdown in convergence rate at later stages
is observed for the FIM-1 and FIM-2 solvers. Notably, the NMG solver
achieves the highest overall performance, characterized by a steepest
residual decay with the lowest computational cost.

\begin{table}[!tbp]
  \tabcolsep=0.2cm
  \centering\small
  \caption{Performance of basic moment and FIM solvers for 
    shock structure computation with $N_1 = 1024$ and $\gamma_2 =
    5$. $T_r$: wall-clock time ratio of FIM-3 to SISGS.}
  \label{tab:shock-SIS-FIM}
  \begin{tabular}{cccccccc}
    \toprule
    \multirow{2}*{$\Ma$} & \multicolumn{3}{c}{Iterations}  &
    \multicolumn{3}{c}{Wall-clock time $(\unit{\s})$} &  \\ 
    \cmidrule(r){2-4} \cmidrule(r){5-7}
         & Euler & SIS & SISGS & Euler & SIS & SISGS &  \\
    \midrule
    1.4 & 18153 & 18154 & 6589 & 3266.70 & 3250.31 & 1808.58 &  \\
    1.55 & 16207 & 16210 & 5642 & 2907.08 & 2900.12 & 1550.99 &  \\ 
    \midrule 
         & FIM-1 & FIM-2 & FIM-3 & FIM-1 & FIM-2 & FIM-3 & $T_r$ \\
    \midrule
    1.4 & 3530 & 3534 & 1298 & 615.97 & 611.61 & 351.09 & 19.4\% \\
    1.55 & 3416 & 3420 & 1180 & 603.24 & 586.52 & 321.95 & 20.8\% \\
      \bottomrule
    \end{tabular}
\end{table}

\begin{table}[!tbp]
  \tabcolsep=0.2cm
  \centering\small
  \caption{Performance of the NMG solver for shock structure
    computation under grid refinement, using FIM-3 as the smoother and
    fixing the coarsest grid to consist of 128 cells. $T_r$: wall-clock
    time ratio of NMG to FIM-3 with $N_1 = 1024$.}
  \label{tab:shock-nmg}
  \begin{tabular}{cccccccccc}
    \toprule
    \multirow{2}*{$\Ma$} & \multicolumn{4}{c}{Iterations} &
    \multicolumn{4}{c}{Wall-clock time $(\unit{\s})$} & \multirow{2}*{$T_r$}\\
    \cmidrule(r){2-5} \cmidrule(r){6-9}
    & 1024 & 2048 & 4096 & 8192 & 1024 & 2048 & 4096 & 8192 & \\
    \midrule
    1.4  & 22 & 21 & 20 & 24 & 37.51 & 68.35 & 131.11 & 318.05 & 10.7 \% \\
    1.55 & 21 & 21 & 21 & 24 & 33.64 & 68.87 & 136.17 & 318.06 & 10.4 \% \\
      \bottomrule
    \end{tabular}
\end{table}

\begin{figure}[!tbp]
  \centering
  \subfloat[$\Ma=1.4$]{
    \label{fig:shock-con-Ma-1.4}
    \includegraphics[width=0.3\textwidth]{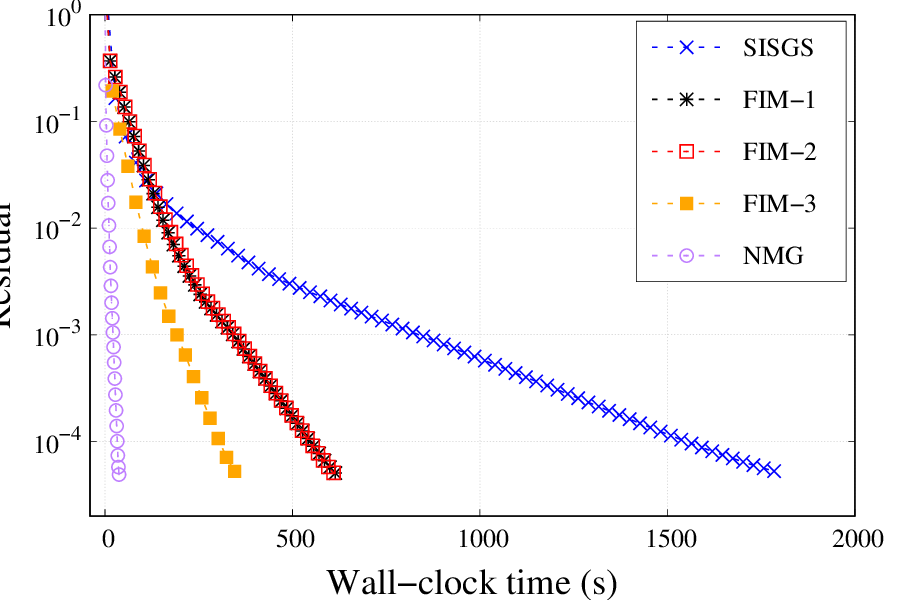}
  }\hfill
  \subfloat[$\Ma=1.55$]{
    \label{fig:shock-con-Ma-1.55}
    \includegraphics[width=0.3\textwidth]{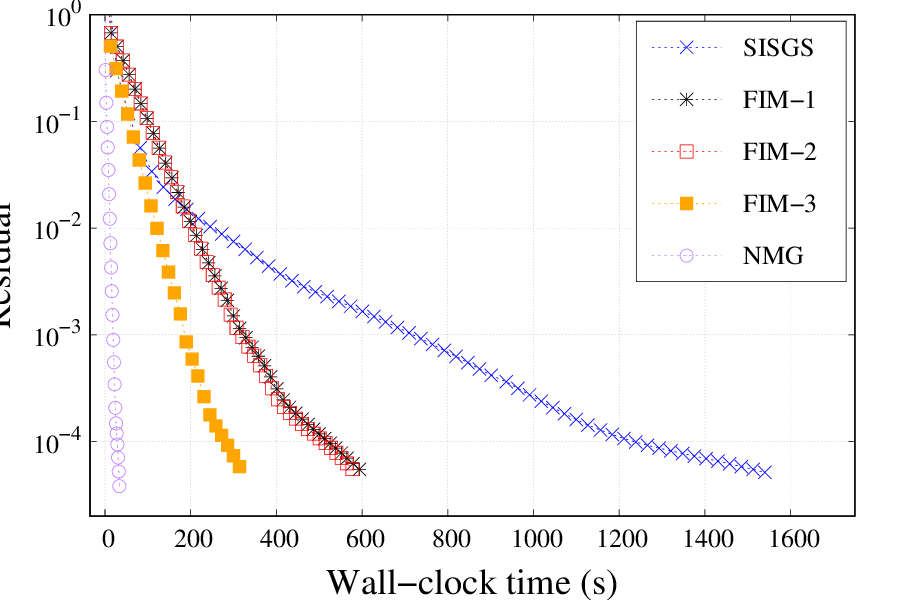}
  }\hfill
  \subfloat[Wall-clock time vs. solver]{
    \label{fig:shock-time-solvers}
    \includegraphics[width=0.3\textwidth]{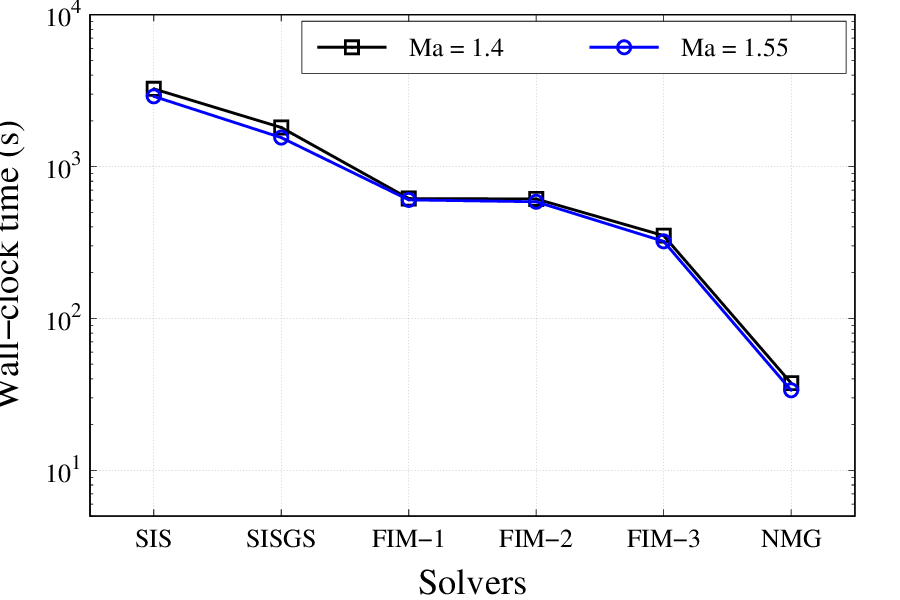}
  }
  \caption{Convergence histories and computational costs of various
    solvers for shock structure computation on the uniform grid
    with $N_1 = 1024$.}
  \label{fig:shock-efficiency}
\end{figure}

\subsection{Lid-driven cavity flow}
\label{sec:cavity}

As one of the classical two-dimensional benchmark problems widely
studied in the literature,
the lid-driven cavity flow is taken as the third example. The same
collision frequency as that given in \cref{eq:nu-couette} for the
Couette flow is employed here. Following the setup explored in
\cite{Cai2018linear, Hu2020, Hu2023}, we consider a square cavity of
side length $L= \qty{1}{\m}$ with wall temperature
$T^{W}=\qty{273}{\kelvin}$. The gas is initially in a uniform
Maxwellian with constant density, mean velocity
$\bu_{0} = \qty[per-mode=symbol]{0}{\m\per \s}$, and temperature
$T_{0}=T^{W} = \qty{273}{\kelvin}$. As the top lid moves horizontally
to the right at a constant speed
$u^{W} = \qty[per-mode=symbol]{50}{\m \per \s}$, the gas evolves and
eventually reaches a steady state. Three initial densities, namely,
$\rho_{0} = \qty[per-mode=symbol]{8.58e-7}{\kg\per\cubic \m}$,
$\qty[per-mode=symbol]{8.58e-6}{\kg\per\cubic \m}$, and
$\qty[per-mode=symbol]{8.58e-5}{\kg\per\cubic \m}$, corresponding to
$\Kn=0.1$, $0.01$, and $0.001$, respectively, are
investigated below.

\subsubsection{Solution validation}
\label{sec:cavity-validation}

Similar to the Couette flow tests, we choose $M=9,10$ for
$\Kn=0.1$, $M=7,8$ for $\Kn=0.01$, and $M=5,6$ for $\Kn=0.001$ in our
simulations. Numerical solutions on the grid with
$N_{1} = N_{2} = 512$, including the temperature field and heat flux
streamlines at $\Kn=0.1$, and the velocity streamlines as well as
normalized velocity profiles along the centerlines at both $\Kn=0.01$
and $0.001$, are presented in \cref{fig:cavity-sol}, where reference
solutions are again obtained from the dugksFoam
solver \cite{Zhu2017foam} for $\Kn=0.1$ and $0.01$, and from the
NSF equations for $\Kn=0.001$. These results show
that the selected orders $M$ are sufficient to produce
solutions in close agreement with the references across all tested
$\Kn$, validating the capability of the moment system
in the near-continuum flow regime using moderate orders.

\begin{figure}[!tb]
  \centering
  \subfloat[Temperature $(\unit{\kelvin})$, $M=10$]{
    \label{fig:cavity_pt1_T}
    \includegraphics[width=0.29\textwidth]{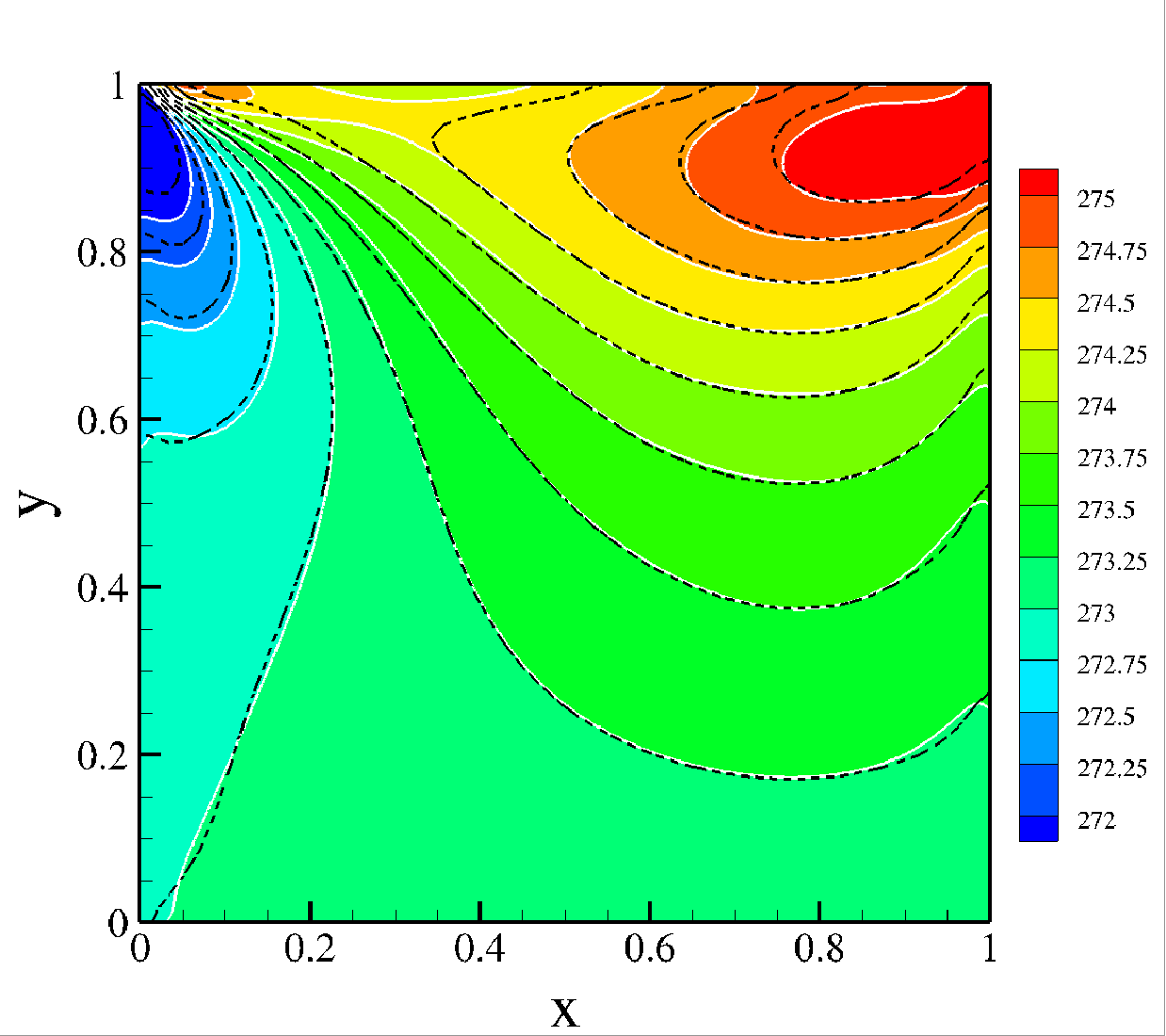}
  }\hfill
  \subfloat[Velocity streamlines, $M=8$]{
    \label{fig:velocity_streamlines_pt01}
    \includegraphics[width=0.29\textwidth]{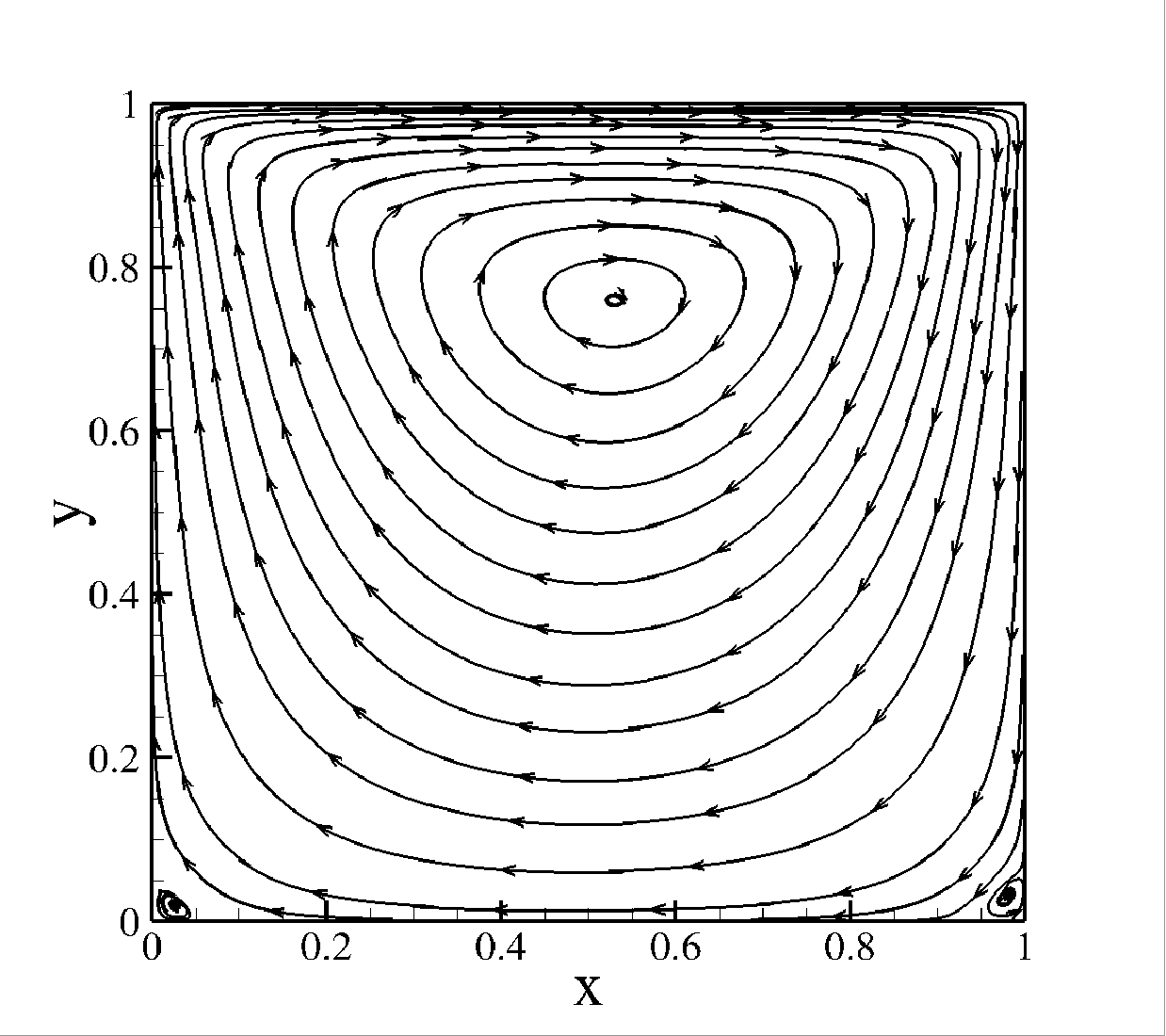}
  }\hfill
  \subfloat[Velocity streamlines, $M=6$]{
    \label{fig:velocity_streamlines_pt001}
    \includegraphics[width=0.29\textwidth]{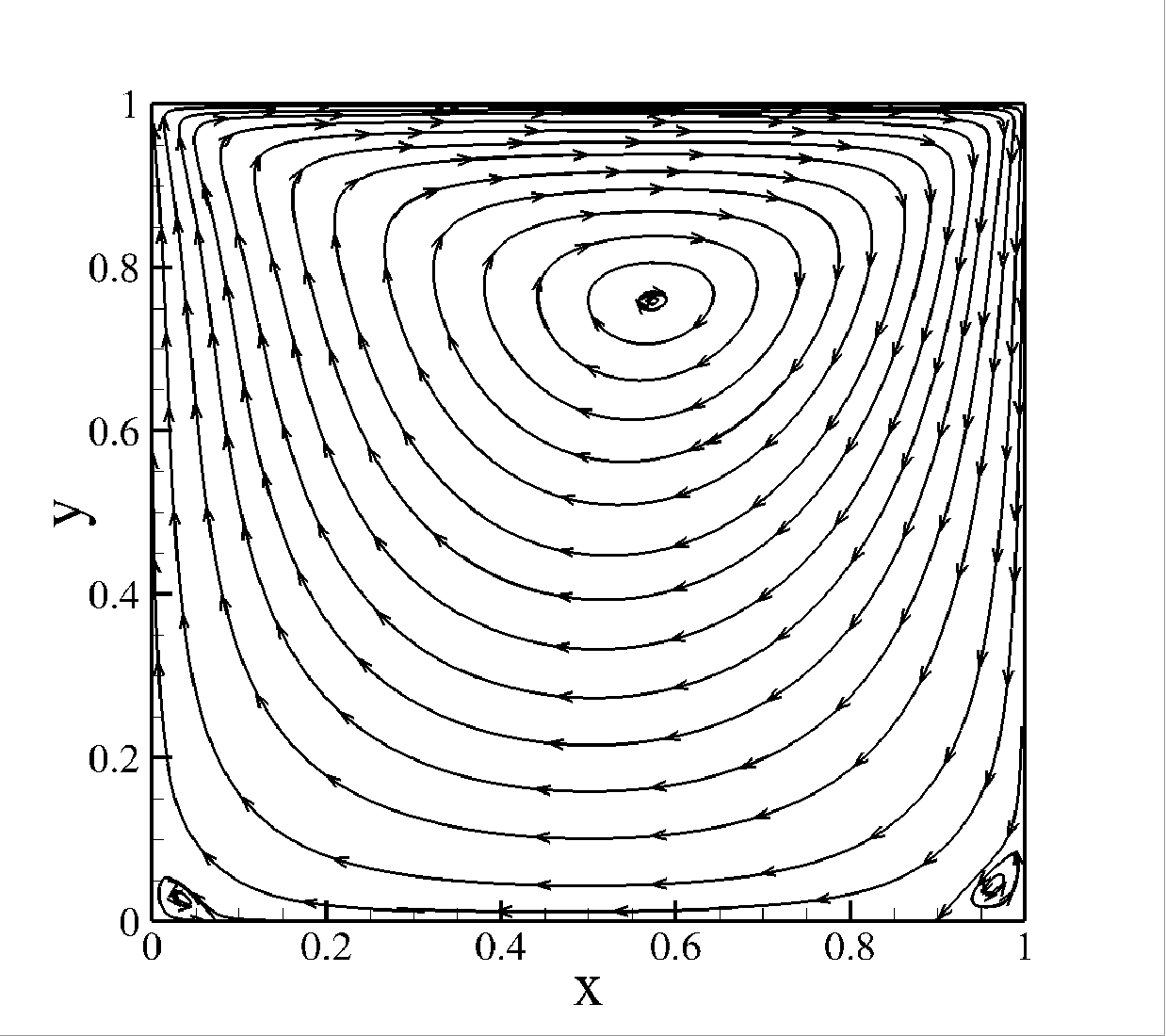}
  }\\
  \subfloat[Heat flux lines, $M=10$]{ 
    \label{fig:cavity_pt1_q}
    \includegraphics[width=0.28\textwidth]{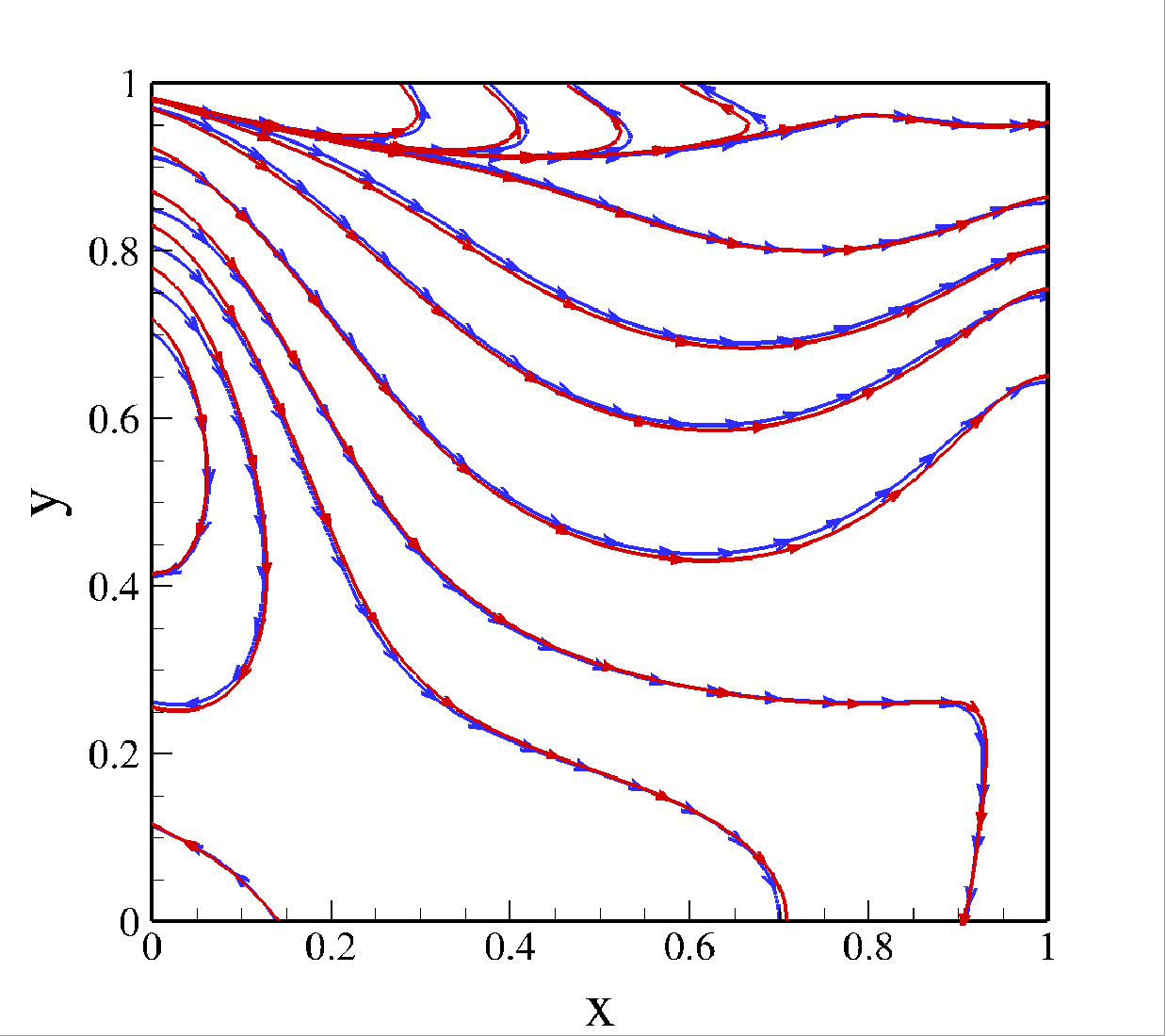}
  }\hfill
  \subfloat[$u_{1}, u_{2}$ along centerlines]{
    \label{fig:slice_u_v_pt01}
    \includegraphics[width=0.29\textwidth]{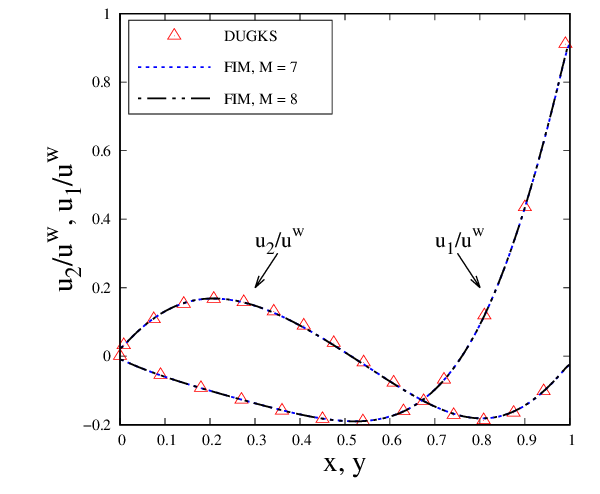}
  }\hfill
  \subfloat[$u_{1}, u_{2}$ along centerlines]{
    \label{fig:slice_u_v_pt001}
    \includegraphics[width=0.29\textwidth]{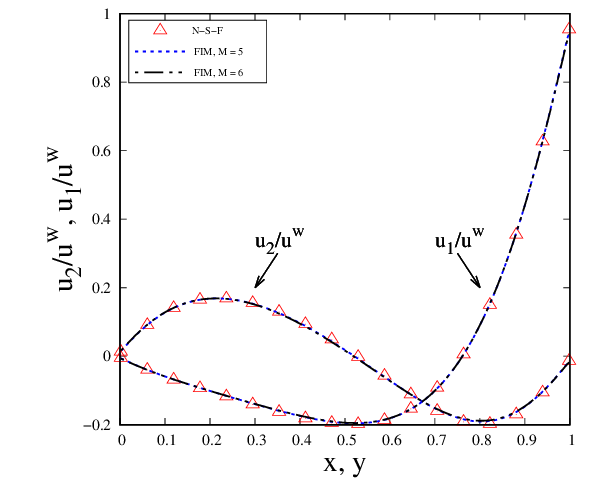}
  }
  \caption{Numerical solutions of lid-driven cavity flow on the
    uniform $512 \times 512$ grid. 
    \emph{Left column}: results at $\Kn=0.1$, with dashed and red lines
    denoting the reference DUGKS solutions; \emph{Middle column}:
    velocity streamlines and normalized velocity profiles ($u_{1}$
    along the vertical centerline, $u_{2}$ along the horizontal
    centerline) at $\Kn=0.01$; \emph{Right column}: corresponding
    velocity results at $\Kn=0.001$.  }
  \label{fig:cavity-sol}
\end{figure}

\subsubsection{Solver performance}
\label{sec:cavity-efficiency}

The performance of the basic SIS and SISGS solvers, along with
their associated FIM-2 and FIM-3 solvers, is first examined.
Following several preliminary tests, as in the previous examples, the
parameter $\gamma_{2}$ in each FIM alternating iteration for the
lid-driven cavity flow is set to $20$, $100$, and $200$ for $\Kn=0.1$,
$0.01$, and $0.001$, respectively. 
The total number of iterations and wall-clock time, consumed by these
solvers on the $128\times 128$ grid, are listed in \cref{tab:cavity_FIM}.

It is observed that the basic SIS and SISGS solvers exhibit the main
behavior analogous to that found in the one-dimensional tests. As
$\Kn$ decreases toward $0.001$, both solvers experience a
significant deterioration in convergence, while the basic SISGS solver
maintains a nearly constant speedup over the SIS solver, with the
wall-clock time reduced by more than half across all tested $\Kn$.
Consequently, at smaller $\Kn$, their total numbers of
iterations increase greatly, incurring considerable computational
cost despite the use of lower-order moment systems. 
In contrast, the corresponding FIM solvers drastically
reduce both the number of iterations and wall-clock time. This
reduction becomes more pronounced as $\Kn$ decreases, resulting in a
remarkable improvement in efficiency. For instance, the FIM-3 solver
requires approximately $530$ iterations to reach the steady state at
$\Kn=0.1$, while this number drops sharply to about $90$ at
$\Kn=0.01$, and further to nearly $50$ at $\Kn=0.001$. Accordingly,
the wall-clock time of the FIM-3 solver relative to the basic SISGS
solver decreases from around $20\%$ at $\Kn=0.1$, to about $5\%$ at
both $\Kn=0.01$ and $0.001$.

\begin{table}[!tbp]
  \tabcolsep=0.1cm
  \centering\small
  \caption{Performance of basic moment and FIM solvers for 
    cavity flow with $N_1 \times N_2 = 128 \times 128$ and
    $\gamma_2 = 20$, $100$, $200$ for $\Kn = 0.1$, $0.01$, $0.001$,
    respectively. $T_r$: wall-clock time ratio of FIM-3 to SISGS.}
  \label{tab:cavity_FIM}
  \begin{tabular}{ccccccccccc}
    \toprule
    \multirow{2}*{$\Kn$} & \multirow{2}*{$M$} & \multicolumn{4}{c}{Iterations}
    & \multicolumn{4}{c}{Wall-clock time $(\unit{\s})$} & \multirow{2}*{$T_r$} \\
    \cmidrule(r){3-6} \cmidrule(r){7-10}
             &  & SIS & FIM-2 & SISGS & FIM-3 & SIS & FIM-2 & SISGS & FIM-3 & \\
    \midrule
    $\num{0.1}$   & 10 & 16154 & 1657 & 4737 & 529 & 86617.47 & 17813.71 & 38079.99 & 7983.22 & 21.0 \% \\
                  & 9  & 15606 & 1652 & 4525 & 517 & 64231.75 & 10591.34 & 28234.65 & 5860.41 & 20.8 \% \\
    $\num{0.01}$  & 8  & 25314 & 240  & 7328 & 92  & 79101.09 & 2798.35  & 35083.23 & 1897.88 & 5.4 \% \\
                  & 7  & 24229 & 232  & 6974 & 87  & 56033.27 & 1865.17  & 24735.46 & 1343.81 & 5.4 \% \\
    $\num{0.001}$ & 6  & 30826 & 196  & 9019 & 54  & 52382.96 & 1976.9   & 22863.31 & 1152.77 & 5.0 \% \\
                  & 5  & 29908 & 190  & 8696 & 52  & 35335.45 & 1626.77  & 15542.41 & 743.58  & 4.8 \% \\
    \bottomrule
  \end{tabular}
\end{table}

The NMG solver using the FIM-3 solver as the smoother is then
evaluated. The
results on three grids with $N_{1}\times N_{2}$ ranging from
$128\times 128$ to $512\times 512$ are summarized in
\cref{tab:cavity_nmg}. In all cases, convergence is achieved within
$30$ iterations, indicating a substantial further reduction in
computational cost, with the wall-clock time ratio of the NMG solver
to the FIM-3 solver on the $128\times 128$ grid being about $15\%$ at
$\Kn=0.1$ and around $50\%$ at $\Kn=0.01$. By contrast, at
$\Kn=0.001$, the NMG solver on the same grid requires only $7$
iterations but consumes the wall-clock time comparable to that of the
FIM-3 solver, offering no clear advantage. Nevertheless, this is not
an intrinsic limitation, as the efficiency can still be effectively
improved through appropriate tuning of the NMG parameters. Indeed, a
supplemental test on this grid for the NMG solver shows that the
wall-clock time decreases from $\qty{728}{\s}$ listed in
\cref{tab:cavity_nmg} for $\gamma_{2}=200$ to around $\qty{400}{\s}$
for a reduced $\gamma_{2} = 100$, despite a slight growth in the total
number of iterations.

As the grid is refined, the total number of iterations incurred by the
NMG solver increases somewhat faster than in the one-dimensional
Couette flow tests. However, as illustrated in
\cref{fig:cavity-ite-grid-nmg}, this growth remains much slower than
that of the FIM-3 solver, even at $\Kn=0.001$, where the total number
of iterations for the NMG solver grows nearly linearly. Therefore, the
efficiency advantage of the NMG solver is expected to persist and
become increasingly significant on finer grids.

In addition, convergence histories of the FIM and NMG solvers, along
with the basic SISGS solver, are plotted in terms of wall-clock time
in \cref{fig:cavity-con-his-fim}. Their computational cost versus $\Kn$
on the $128\times 128$ grid with $M=5$ is presented
in \cref{fig:cavity-con-his-nmg}. All results exhibit steady and
efficient residual decay for the FIM and NMG solvers, demonstrating
their robustness and superior efficiency. While the NMG solver is also
observed to become less efficient as $\Kn$ decreases, and even to show
little to no speedup over the FIM-3 solver at $\Kn=0.001$, it is worth
reiterating that this shortcoming can be effectively mitigated by
carefully optimizing the parameters in the solver.

\begin{table}[!tbp]
  \tabcolsep=0.1cm
  \centering\small
  \caption{Performance of the NMG solver for cavity flow under grid
    refinement, using FIM-3 solver as the smoother. $T_{r}$:
    wall-clock time ratio of NMG to FIM-3 with
    $N_1 \times N_2 = 128 \times 128$.}
  \label{tab:cavity_nmg}
  \begin{tabular}{cccccccccc}
    \toprule
    \multirow{2}*{$\Kn$} & \multirow{2}*{$M$}  & \multicolumn{3}{c}{Iterations}
    & \multicolumn{3}{c}{Wall-clock time $(\unit{\s})$} & \multirow{2}*{$T_r$} \\
    \cmidrule(r){3-5} \cmidrule(r){6-8}
    & & $128\times 128$ & $256\times 256$ & $512\times 512$ & $128\times 128$ & $256\times 256$ & $512\times 512$ \\
    \midrule
    $\num{0.1}$   & 10 & 16 & 22 & 30 & 1185.84 & 5086.56 & 24400.10 & 14.85 \%\\
                  & 9  & 15 & 21 & 28 & 839.69  & 3644.39 & 17495.09 & 14.33 \%\\
    $\num{0.01}$  & 8  & 8  & 12 & 16 & 966.76  & 4075.36 & 16483.24 & 50.94 \%\\
                  & 7  & 7  & 11 & 15 & 620.64  & 2971.82 & 13155.98 & 46.19 \%\\
    $\num{0.001}$ & 6  & 7  & 12 & 22 & 963.00  & 5420.30 & 34310.46 & 83.54 \%\\
                  & 5  & 7  & 12 & 21 & 728.63  & 4330.57 & 28040.09 & 97.99 \%\\
    \bottomrule
  \end{tabular}
\end{table}

\begin{figure}[!tbp]
  \centering
  \subfloat[Iterations vs. grid cells]{
    \label{fig:cavity-ite-grid-nmg}
    \includegraphics[width=0.39\textwidth]{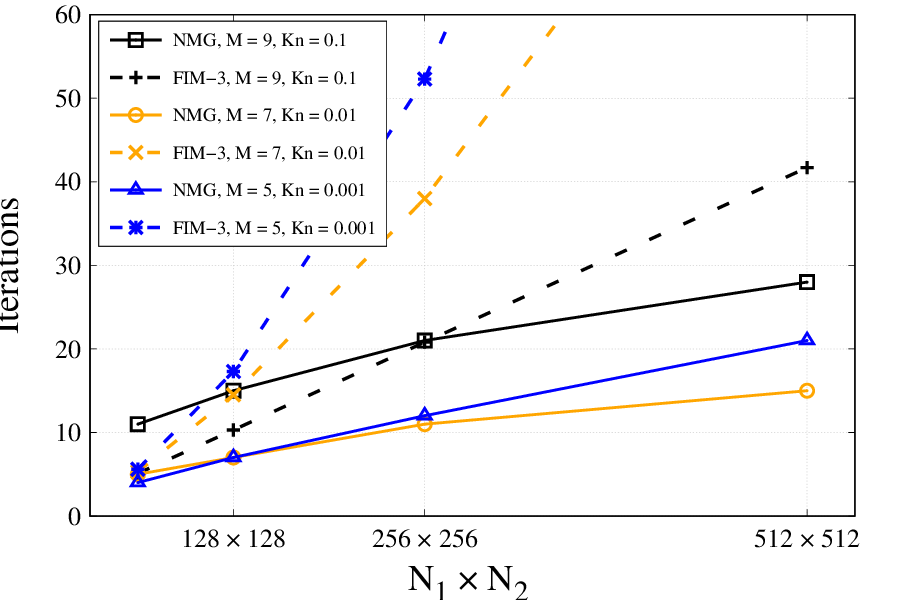}
  }\qquad
  \subfloat[$M=5$, $N_{1}=N_{2} = 128$]{
    \label{fig:cavity-con-his-nmg}
    \includegraphics[width=0.39\textwidth]{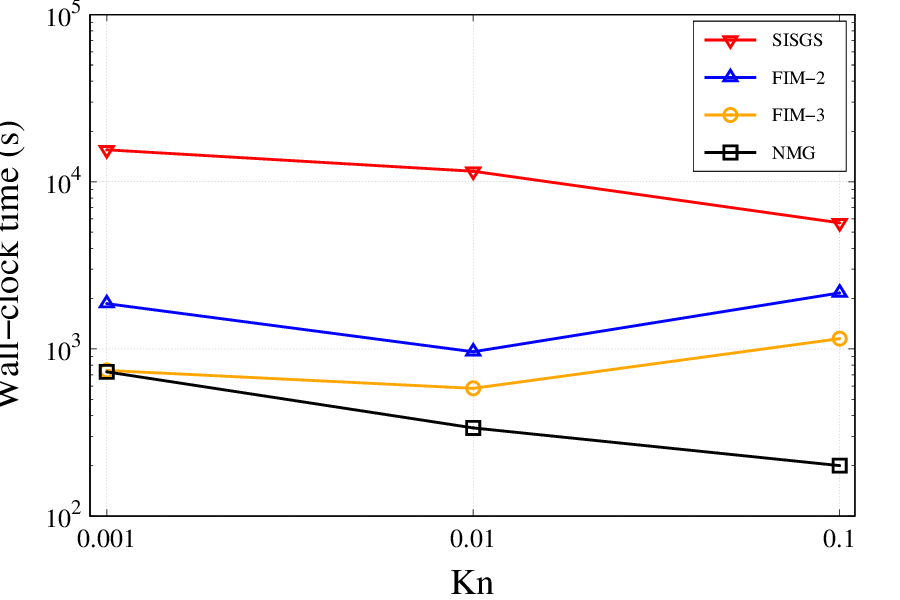}
  }
  \caption{Performance of various solvers for cavity flow. The number
    of iterations for the FIM-3 solver is normalized by ${N}_{t}=50$,
    $6$, and $3$ for $\Kn=0.1$, $0.01$, and $0.001$, respectively.}
  \label{fig:cavity-nmg-convergence}
\end{figure}

\begin{figure}[!tbp]
  \centering
  \subfloat[$\Kn=0.1$, $M=9$]{
    \includegraphics[width=0.31\textwidth]{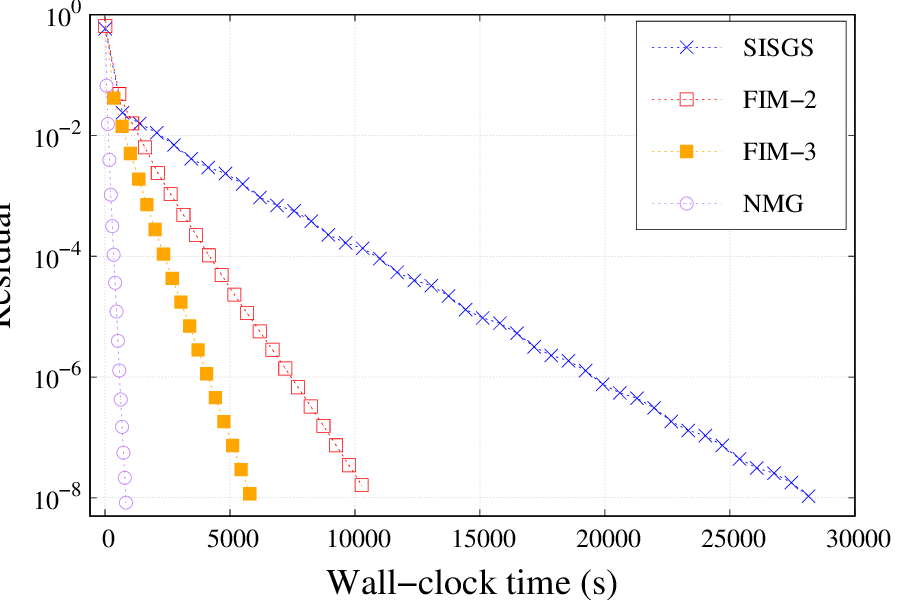}
  }\hfill
  \subfloat[$\Kn=0.01$, $M=7$]{
    \includegraphics[width=0.31\textwidth]{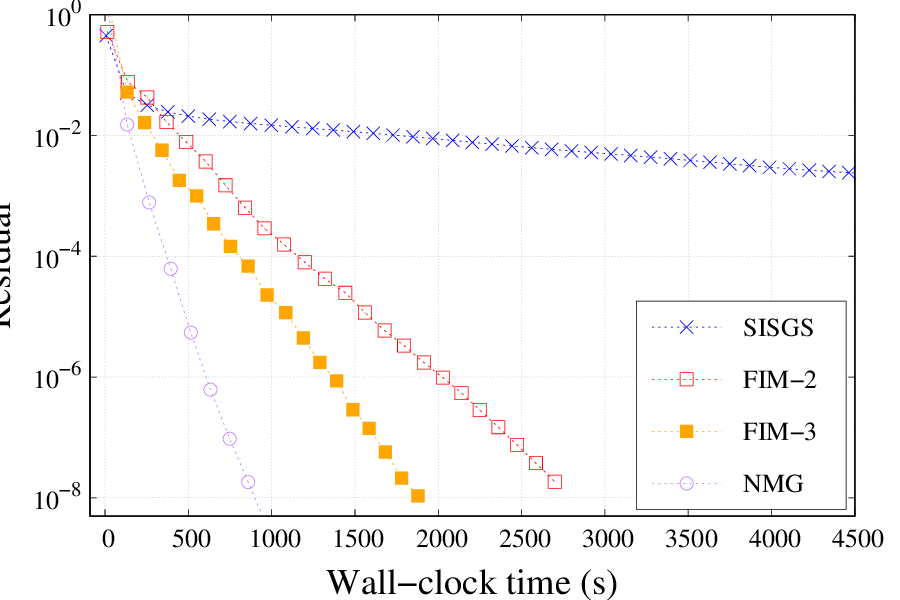}
  }\hfill
  \subfloat[$\Kn=0.001$, $M=5$]{
    \includegraphics[width=0.31\textwidth]{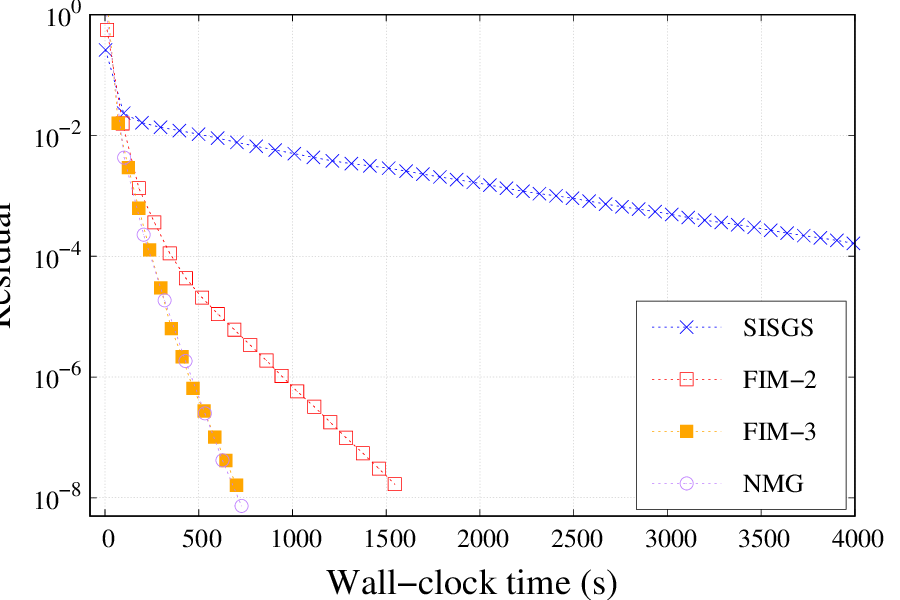}
  }
  \caption{Convergence histories of various solvers for
    cavity flow on the uniform $128\times 128$ grid.
  }
  \label{fig:cavity-con-his-fim}
\end{figure}


\section{Conclusion}
\label{sec:conclusion}

A novel fast iterative moment (FIM) method, which alternately solves
the high-order moment system and the consistent hydrodynamic
equations, has been developed for the steady-state Boltzmann-BGK equation.
Benefiting from the use of the hydrodynamic equations to efficiently
evolve the primary macroscopic quantities, the FIM method
significantly accelerates convergence in the near-continuum flow regime,
with such an acceleration becoming increasingly pronounced as
$\Kn$ decreases. To improve robustness and efficiency,
several targeted numerical strategies, including a semi-implicit
scheme for the moment system and a symmetric Gauss-Seidel method for
both the moment system and the hydrodynamic equations, have been
incorporated into the FIM framework, yielding a family of FIM solvers
with notably enhanced performance. It is also noteworthy that the
present FIM method can be interpreted as a redesigned two-level NMLM
method with improved robustness. Unlike the original NMLM method,
however, this enhanced robustness enables the FIM method to be
successfully incorporated into the spatial NMG framework, resulting in
a more efficient multi-level solver that integrates coarse level
corrections in both spatial and velocity spaces. Numerical
experiments on planar Couette flow, shock structure, and
two-dimensional lid-driven cavity flow validate the effectiveness of
the proposed FIM and FIM-based NMG solvers, confirming their
robustness, efficiency and favorable scaling behavior in the
near-continuum flow regime, particularly as $\Kn$ decreases or the
grid is refined. Moreover, the FIM method allows for a flexible choice
of order $M$, as the additional computational cost remains moderate
even for a relatively high-order moment system.

Finally, to further enhance the performance of the FIM method,
possible extensions, such as incorporating high-order spatial
discretizations and Newton-type iterations, are currently under
investigation and will be reported in forthcoming work.


\bibliographystyle{siamplain}
\bibliography{reference.bib}
\end{document}